\documentclass[leqno,11pt]{amsart}
\usepackage[margin=1in,marginparwidth=75pt,centering]{geometry}
\usepackage{mathtools,mathrsfs,lineno,paralist,graphicx,float}
\usepackage{amssymb,amsthm,amsmath}
\usepackage[T1]{fontenc}
\usepackage[utf8]{inputenc}
\usepackage{color}
\usepackage[svgnames]{xcolor}
\usepackage{ifpdf}
\usepackage[all]{xy}

\usepackage{tikz-cd} 
\usepackage{bbm}

\ifpdf \usepackage[pdftex,pdfstartview=FitH,pdfpagemode=none,colorlinks,bookmarks,linkcolor=blue,
citecolor=Indigo
]{hyperref} \else  \usepackage[hypertex]{hyperref} \fi

\usepackage{pgfplots}
\usepackage{tikz}
\usetikzlibrary{arrows}
\usepackage{comment}
\usepackage{bm}
\usepackage{ulem}
\usepackage{enumerate}
\usepackage{enumitem}
\usepackage{mfirstuc}

\definecolor{fadeblue}{RGB}{0,57,128}
\def\fadeblue{\color{fadeblue}}
\usepackage{etoolbox}
\patchcmd{\section}{\normalfont}{\normalfont \fadeblue}{}{}
\patchcmd{\subsection}{\normalfont}{\normalfont \fadeblue}{}{}
\patchcmd{\subsubsection}{\normalfont}{\normalfont \fadeblue}{}{}

\usepackage{calc}
\newcommand{\Aa}{{\overset{\circ}{A}}}
\renewcommand{\leq}{\leqslant}
\renewcommand{\geq}{\geqslant}
\newcommand{\hide}[1]{}

\newcommand{\rmm}[1]{\mathrm{#1}} 
\newcommand\set[1]{\left\{#1\right\}} 
\newcommand\pa[1]{\left(#1\right)}

\newcommand\av[1]{\left|#1\right|}
\newcommand\op[1]{\operatorname{#1}}
\newcommand\diag[1]{\operatorname{diag}\left(#1\right)}

\newcommand{\schwartzmanGroup}{{\mathfrak{A}}}
\DeclareMathOperator{\supp}{supp}

\author{Xianzhe Li}
\address{Chern Institute of Mathematics and LPMC, Nankai University, Tianjin 300071, China} 
\address{Department of Mathematics, University of California, Berkeley, CA 94720, USA} 
 \email{xianzhe@berkeley.edu}
 
 \author {Li Wu}
 \address{
    School of Mathematics, Shandong University, Jinan, China} 
 \email{vvvli@sdu.edu.cn}

\newtheorem{theorem}{Theorem}[section]
\newtheorem{lemma}[theorem]{Lemma}
\newtheorem{proposition}[theorem]{Proposition}
\newtheorem{corollary}[theorem]{Corollary}

\theoremstyle{definition}

\newtheorem{definition}[theorem]{Definition}
\newtheorem{example}[theorem]{Example}
\newtheorem{remark}[theorem]{Remark}

\theoremstyle{plain}

\numberwithin{equation}{section}

\newcommand{\A}{{\mathbb A}}
\newcommand{\D}{{\mathbb D}}

\newcommand{\F}{{\mathbb F}}

\newcommand{\N}{{\mathbb N}}

\newcommand{\R}{{\mathbb R}}
\newcommand{\C}{\mathbb{C}}
\newcommand{\T}{{\mathbb T}}

\newcommand{\Z}{{\mathbb Z}}


\makeatletter
\def\saveenum{\xdef\@savedenum{\the\c@enumi\relax}}
\def\resetenum{\global\c@enumi\@savedenum}
\makeatother

\begin{document}
\title{The fibered rotation number for ergodic symplectic cocycles
and its applications:
I. Gap Labelling Theorem}
\maketitle
\begin{abstract}
	Let $ (\Theta,T,\mu) $ be an ergodic topological dynamical system. The fibered rotation number for cocycles in $ \Theta\times \rmm{SL}(2,\R) $, acting on $ \Theta\times \R\mathbb{P}^1
	 $ is well-defined and has wide applications in the study of the spectral theory of Schr\"odinger operators. In this paper, we will provide its natural generalization for higher dimensional cocycles in $ \Theta\times\rmm{Sp}(2m,\R) $ or $ \Theta\times \rmm{HSp}(2m,\C) $, where $ \rmm{Sp}(2m,\R) $ and $ \rmm{HSp}(2m,\C) $ respectively refer to the $ 2m $-dimensional symplectic or Hermitian symplectic matrices. 
	As a corollary, we establish the equivalence between the integrated density of states for generalized Schr\"odinger operators and the fibered rotation number; and  the Gap Labelling Theorem via the Schwartzman group, as expected from the one dimensional case \cite{AS1983Almost,gaplabel}. 
\end{abstract}
\tableofcontents
\pagestyle{plain}
\section{Introduction}

We consider the generalized Schr\"odinger operator $ H $ acting on $ \ell^2(\Z,\C^m) $,
\begin{equation} \label{eq:H}
    (H\textbf{u})_n=C^* \textbf{u}_{n-1}+B(n)\textbf{u}_n+C\textbf{u}_{n+1},
\end{equation} 
where $ \textbf{u}_n\in \C^m $, $ C\in \rmm{GL}(m,\C) $, $ B(n)\in \rmm{Her}_m\C	 $, the set of $ m\times m $ Hermitian matrices over $ \C $. 
The main case of interest is that the diagonal block, $ B $, is dynamically defined over a {\it topological dynamical system}. That is, given a compact metric space $ \Theta $, a homeomorphism $ T:\Theta\to\Theta $, and a continuous matrix-valued function $ f:\Theta\to \rmm{Her}_m\C  $, we consider
\begin{equation} \label{eq:H2}
	(H_{\theta}\textbf{u})_n=C^* \textbf{u}_{n-1}+B_{\theta}(n)\textbf{u}_n+C\textbf{u}_{n+1},
\end{equation} 
where $ B_{\theta} $ is defined by  $B_{\theta}(n)=f(T^n\theta) $.

Dynamically defined generalized Schr\"odinger operators have been extensively studied over the years, with particular significance attached to the one-dimensional case ($m=1$). This specific scenario corresponds to one-dimensional Schr\"odinger operators acting on $\ell^2(\mathbb{Z})$, an area that has seen remarkable advancements in understanding spectral properties \cite{avila2008absolutely,AJ3,avila,AYZ,avila2016dry}.

A central object of study within the $m=1$ context is the quasi-periodic Schr\"odinger operator:
\begin{equation}\label{quasi-periodic}
    (H_{ v,\alpha,\theta}u)_n = u_{n+1} + u_{n-1} + v(\theta + n\alpha)u_n, \quad n \in \mathbb{Z},
\end{equation}
where $v \in C^0(\mathbb{T}^d, \mathbb{R})$ is the potential, \(\theta \in \mathbb{T}^d\) is the phase, and \(\alpha \in \mathbb{T}^d\) is a rationally independent frequency vector. These operators hold profound connections to condensed matter physics and dynamical systems, with comprehensive reviews available in \cite{Dam2017Schrödinger,DF,DF2,J,Y}.

A key motivation for studying generalized Schr\"odinger operators lies in their relation to Aubry duality. Crucially, the generalized operator studied in this paper, specifically the form denoted by \eqref{eq:H}, naturally arises as the Aubry dual of the quasi-periodic operator \eqref{quasi-periodic}. Aubry duality has been instrumental in analyzing fundamental phenomena such as localization-delocalization transitions \cite{BJ,Puig11,AJ2,JK,AYZ,jitomirskaya2021point,GJ} and the Cantor spectrum problem \cite{Puig11,Puig1,AJ3,GJY,avila2016dry}, solidifying its role as a pivotal tool in the spectral theory of quasi-periodic operators.

While the $m=1$ case is well-explored, the generalization to $m>1$ is more than just a mathematical extension; it represents a physically realistic framework. These higher-dimensional generalized operators appear in various physical models and theoretical studies \cite{bid,Rod,Sar}, offering insights into more complex lattice structures or multi-particle interactions that extend beyond one-dimensional descriptions. Therefore, understanding their properties is essential for a broader range of physical applications.


In this paper, we will always assume that  equip the topological dynamical system $ (\Theta,T) $ with a $ T $-ergodic Borel probability measure $ \mu $. We will simply denote by $ (\Theta, T, \mu) $. The ergodic assumption is crucial, as it ensures that spectral properties derived from a single typical realization of the potential are representative of the entire ensemble, providing a robust statistical framework for our analysis.

Given $ (\Theta, T, \mu) $, consider the family of operators $ \{H_{\theta}\}_{\theta\in\Theta} $ defined as above. It follows that the spectrum is $ \mu $-almost everywhere constant; compare \cite{DF}. That is, there exists a compact set $ \Sigma_{\mu}\subseteq \R $ such that 
\[
	\sigma(H_{\theta})=\Sigma_{\mu}, \text{ for }\mu\text{-a.e.\ }\theta \in \Theta.
\] 
We designate $\Sigma_{\mu}$ as the {\it almost sure spectrum} of the family $ \{H_{\theta}\}_{\theta\in\Theta} $. 
If $(\Omega,T)$ is uniquely ergodic, then we simply write $\Sigma = \Sigma_{\mu}$ where $\mu$ is the unique invariant measure. Similarly, if $(\Omega,T)$ is minimal, then there exists $\Sigma $ such that $\sigma(H_{\theta}) = \Sigma $ for {\it{all}} $\theta \in \Theta $ and again there is no need to note the dependence on the measure. Note that the quasiperiodic action on torus $ (\T^d,\alpha): \theta\mapsto \theta+\alpha $ with $ \alpha\in\R^d $ rationally independent, is strictly ergodic, that is, uniquely ergodic and minimal, the associated operator $H_{\alpha,\theta}$ exhibits $\theta$-independent spectrum. 

The complement of the almost sure spectrum is called the {\it spectral gaps}, which, in the setting of dynamical systems, are closely related to {\it uniform hyperbolicity}: Given \((\Theta,T,\mu)\), if the orbit of \(\theta_0\) is dense in \(\Theta\), then \(E\) falls into the spectral gaps of \(H_{\theta_0}\) if and only if \((T,A_E)\) is uniformly hyperbolic, where \((T,A_E)\) is the so-called {\it Schrödinger cocycle}.

Let \( \omega_m \) denote the standard symplectic structure
\[
\omega_m = 
\begin{pmatrix} 
0 & -I_m \\ 
I_m & 0 
\end{pmatrix},
\]
and let \( \operatorname{Lag}(\mathbb{C}^{2m},\omega_m) \) denote the {\it Lagrangian Grassmannian} of \( (\mathbb{C}^{2m},\omega_{m}) \). A basic fact in symplectic geometry is that the action of \( \mathrm{HSp}(2m,\mathbb{C}) \) preserves \( \operatorname{Lag}(\mathbb{C}^{2m},\omega_m) \). 
Therefore, given any \( A \in C^0(\Theta, \mathrm{HSp}(2m,\mathbb{C})) \), the {\it Hermitian symplectic cocycle} \( (T,A) \) is defined on \( \Theta \times \operatorname{Lag}(\mathbb{C}^{2m},\omega_m) \)  as the skew-product map
\[
(T, A) : \Theta \times \operatorname{Lag}(\mathbb{C}^{2m}, \omega_m) \to \Theta \times \operatorname{Lag}(\mathbb{C}^{2m}, \omega_m), \quad (\theta, \Lambda) \mapsto \left( T\theta, A(\theta)\Lambda \right).
\]

Write the eigenvalue equation \(H_{\theta}\textbf{u}=E\textbf{u}\) in matrix form as:
\begin{equation*} 
	H_{\theta}=\begin{pmatrix}
	    \ddots & \ddots &\ddots & \\
		& C & B_{\theta}(n+1)& C^*&\\
		& & C & B_{\theta}(n)& C^*\\
		& & & C & B_{\theta}(n-1)& C^* & \\
		& & & & \ddots & \ddots &\ddots & \\
	\end{pmatrix}.
\end{equation*}
This tridiagonal structure leads to the following recurrence relation:
\begin{equation}
 	\begin{pmatrix}
 	    \textbf{u}_{n+1}\\
 	    \textbf{u}_n
 	\end{pmatrix}=\widehat{A}_E(T^n\theta)\begin{pmatrix}
 	    \textbf{u}_{n}\\
 	    \textbf{u}_{n-1}
 	\end{pmatrix}, \qquad \widehat{A}_E(\cdot)=\begin{pmatrix}
		C^{-1}(E-f(\cdot ))&-C^{-1}C^*\\
		I_m& 0
	\end{pmatrix}.
\end{equation}
We conjugate \(\widehat{A}_E(\cdot)\) by \(P=\begin{psmallmatrix} C & \\ & I_m \end{psmallmatrix}\), which yields
\begin{equation} \label{eq:hspstruc}
	A_E:=P\widehat{A}_E(\cdot)P^{-1}=\begin{pmatrix}
		(E-f(\cdot))C^{-1}& -C^*\\
		C^{-1}& 0
	\end{pmatrix}\in C^0(\Theta,\rmm{HSp}(2m,\C)).
\end{equation}
The pair \((T,A_E)\) is then called the Schrödinger cocycle\footnote{Actually, \eqref{eq:hspstruc} means \(\widehat{A}_E\) is Hermitian symplectic under the symplectic structure \(\tilde{\omega}_m=\begin{psmallmatrix} 0 & -C^* \\ C & 0 \end{psmallmatrix}\). Then, \((T,\widehat{A}_E)\) is defined on \(\Theta\times \op{Lag}( \mathbb{C}^{2m},\tilde{\omega}_m )\). It is often convenient to replace \((T,\widehat{A}_E)\) with \((T,A_E)\), since there is no essential difference between their dynamical properties.}. 

From a dynamical systems perspective, the eigenvalue equation can be recast as a cocycle — an iterated map that describes the evolution of solutions from one lattice site to the next — acting on a space that captures the essential behavior of these solutions.
The Lagrangian Grassmannian then serves as a natural geometric space in which these solutions `live' or are tracked, allowing us to visualize how their properties — such as growth or rotation — evolve under iteration. This dynamical behavior is directly linked to the spectral properties of the operator. More precisely, the \textit{integrated density of states} (IDS), a fundamental spectral descriptor of~$H_\theta$.

Let $ \delta_{n,i} \in  \ell^2(\mathbb{Z},\mathbb{C}^m) $ be defined by
\[
	(\delta_{n,i})_l = 
	\begin{cases}
		\mathbf{0} \in \mathbb{C}^m, & \text{if } l \neq n, \\
		\mathbf{e}_i \in \mathbb{C}^m, & \text{if } l = n,
	\end{cases}
\]
where $\{ \mathbf{e}_i \}_{i=1}^m$ is the standard basis of $\mathbb{C}^m$. 
The \textit{density of states measure} (DOSM) is the measure $k = k_\mu$ defined by
\[
	\int_{\mathbb{R}} g \, dk = \int_{\Theta} \frac{1}{m} \sum_{i=1}^m \langle \delta_{0,i}, \, g(H_{\theta}) \delta_{0,i} \rangle \, d\mu(\theta)
\]
for bounded measurable $g$.
The IDS is then given by the cumulative distribution function
\begin{equation} \label{k_E}
	\mathcal{N}(E) = \mathcal{N}_\mu(E) = \int \chi_{(-\infty,E]} \, dk.
\end{equation}
Physically, and 
in fact equivalently, the IDS represents the asymptotic distribution of eigenvalues of large finite-volume restrictions of \( H_\theta \).
This direct relationship highlights how the dynamical behavior observed in the Lagrangian Grassmannian yields key insights into the spectrum.

\subsection{Gap labels and Gap-Labelling Theorem}
An important fact is that
\[
	\Sigma_{\mu} = \supp k.
\]
This means that the almost sure spectrum coincides precisely with the set of increase points of $\mathcal{N}$. Thus, the IDS, $\mathcal{N}(E)$, is constant on each spectral gap, i.e., the connected component of $\mathbb{R} \setminus \Sigma_{\mu}$. These constant values are called \textit{gap labels}.

It is thus completely natural to compute or characterize these gap labels, which serves as a long time problem in this field.  We are now prepared to present the results in our setting:

\begin{theorem}\label{t:gablabel}

	Given $(\Theta,T,\mu)$ with $\supp\mu = \Theta$. Let $ \{H_\theta\}_{\theta\in\Theta} $ be a family of operators defined as above. Then for all $E \in \R \setminus \Sigma_{\mu}$, we have
	\begin{equation}
		m \mathcal{N}_{\mu}(E) \in \schwartzmanGroup(\Theta,T,\mu) \cap [0,m], 
	\end{equation} 
	where $ \schwartzmanGroup(\Theta,T,\mu)\subseteq \R $ denotes the Schwartzman group of $ (\Theta,T,\mu) $.
	\end{theorem}

\begin{remark}
The optimality of gap labeling is well illustrated by models with base dynamics such as hyperbolic toral automorphisms (e.g., the cat map on $\mathbb{T}^2$) or the doubling map on the half-strip. In these cases, the Schwartzman group is $\mathbb{Z}$, which implies that the spectrum of $H_{\theta}$ has at most $m$ intervals \cite{DF2023,DEF2023}. This $m$-interval bound is sharp, confirmed by using constant potentials with $m$ well-separated eigenvalues.
\end{remark}

	\begin{example}\label{Gap-Label}
		Suppose $ (\Theta,T,\mu)=(\T^d,\alpha,\op{Leb}) $. Then for any $ E \in \R\setminus \Sigma $, we have 
		\[
			m\mathcal{N}(E)\in \schwartzmanGroup(\Theta,T,\mu)\cap [0,m]=\Z^d\alpha+\Z.
		\] 
	\end{example}

	\begin{remark}\label{fractional-label}
		In other words, fractional gap labels are permitted. However, as we will see, this necessarily occurs in genuinely high-dimensional cases. We discuss this further in Appendix \ref{apx.aubryD}.
	\end{remark}

\begin{remark}
	For the reader's convenience, let us briefly recapitulate the construction of the Schwartzman group $\schwartzmanGroup(\Theta,T,\mu)$.
	Let $(X,\overline{T},\overline{\mu})$ denote the suspension of $ (\Theta,T,\mu) $, which is given by 
	\[
		X=\Theta\times [0,1]/ ((\theta,1)\sim(T\theta,0)),
	\] 
	\[
		\int_{X}f d\overline{\mu}=\int_{\Theta}\int_{0}^{1}f([\theta,t])dt d\mu(\theta),
	\] 
	and $ \overline{T} $ denotes the horizontal flow in the second factor of $ X $. That is, $\overline{T}^t[\theta,s] = [\theta,s+t]$.	
	Let $C^\sharp(X,\T)$ denote the set of homotopy classes of maps $X \to \T$. Then the Schwartzman homomorphism $\mathfrak{A}_{\overline{\mu}}: C^\sharp(X,\T) \to \R$ is defined by 
	\[
		\mathfrak{A}_{\overline{\mu}}([\phi])=\lim_{t\to\infty} \frac{\tilde{\phi}_x(t)}{t}, \ \overline{\mu}\text{-a.e. } x\in X,
	\] 
	where $ \tilde{\phi}_x(t) $ is any lift of $ \phi_x:t\mapsto \phi(\overline{T}^t x) $ to a map $ \R\to\R $. The Schwartzman group is then the range of this homomorphism, $ \schwartzmanGroup(\Theta,T,\mu) := \mathfrak{A}_{\overline{\mu}}(C^\sharp(X,\T)) $.
\end{remark}
\begin{remark}\label{rem:Alter}
	In cases where the base dynamics are sufficiently simple, as illustrated in Example \ref{Gap-Label}, one can provide an alternative proof that does not rely on the notion of the Schwartzman group. Please refer to Appendix \ref{apx.Ex1} for further details.
\end{remark}
The approach of labeling the gaps using the Schwartzman group was first developed by Johnson in the case $m=1$ \cite{Joh1986Exponential}. A detailed proof is available in \cite{DF}. This method hinges on the crucial observation that for every energy outside the spectrum, there exist continuous invariant sections for the associated energy-dependent cocycle. This key insight, now known as Johnson's theorem, characterizes the complement of the almost sure spectrum as the set of energies where the associated cocycles exhibit uniform hyperbolicity. Haro-Puig \cite{HP2013Thouless} later extended this approach to higher dimensions.

\subsection{``Fibered rotation number equals IDS''}

It is a folklore fact that the integrated density of states is actually equal to the {\it fibered rotation number}, the latter of which arises naturally from the underlying dynamical system, known as the Schrödinger cocycle. This equivalence is well established in the one-dimensional case ($m=1$), as detailed in \cite{gaplabel,AS1983Almost}. However, in the high-dimensional case ($m>1$), the relationship is more indirect, and one missing link remains. We address this issue in the present paper.

Let us briefly explain the concept of the fibered rotation number, more details will be provided in the following sections.
 Given a Hermitian symplectic cocycle $ (T,A) $ on $ \Theta\times \op{Lag}( \mathbb{C}^{2m},\omega_m ) $: $ (\theta,\Lambda)\mapsto (T\theta,A(\theta)\Lambda) $. An equivalent and maybe more convenient way to consider the dynamical properties of $ (T,A) $ is  viewing it as acting on $ \Theta\times \mathbf{U}(m): (\theta, W_{\Lambda})\mapsto(T\theta, W_{A(\theta)\Lambda}) $, where $ W_{\Lambda}=(X+iY)(X-iY)^{-1} $, for  $ \Lambda=\begin{bsmallmatrix}
	X\\Y
  \end{bsmallmatrix}\in \op{Lag}( \mathbb{C}^{2m},\omega_m )  $. 
\( W_{\Lambda} \) is isomorphic and independent of the choice of Lagrangian frame.

Now suppose $ A $ is {\it homotopic to the identity}, there exists a continuous map $ \tilde{F}_{T,A}$ acting on the covering space $\Theta\times \R\times \mathbf{SU}(m)  $, of the form $ \tilde{F}_{T,A}(\theta,x,S)=(T\theta,x+f(\theta,x,S),*) $, such that $ f(\theta,x,S)=f(\theta,x',S') $ whenever $ (x,S) $, $ (x',S')\in \R\times \mathbf{SU}(m) $ projects to the same point $ W_{\Lambda}\in \mathbf{U}(m) $. In order to simplify the terminology we shall say that $  \tilde{F}_{T,A} $ is a {\it lift} for $ (T,A) $. The map $ f $ is then descends to a map $ \Theta\times \mathbf{U}(m) \to \R $ and is independent of the choice of the lift, up to the addition of a continuous integer-valued function $ p(\theta)\in C(\Theta,\Z) $. 
Denote 
\[
	\mu(C(\Theta,\Z)):=\left\{ \int p\  d\mu:\  p\in C(\Theta,\Z)  \right\}.
\] 
   Then the limit 
  \[
	\lim_{n\to +\infty}\frac{1}{n}\sum_{k=0}^{n-1} f\pa{(T,A)^k(\theta,W_{\Lambda})} \mod \mu(C(\Theta,\Z))
  \] 
  exists for $ \mu $-a.e. $ \theta\in\Theta $ and all $ W_{\Lambda}\in\mathbf{U}(m) $, and coincides with 
  \[
	\op{rot}_f(T,A)=\int_{\Theta\times \mathbf{U}(m)} f(\theta,W_{\Lambda})d\nu \mod \mu(C(\Theta,\Z))
  \] 
  where $ \nu $ is any probability measure which is invariant under $ (T,A) $ and which projects to $ \mu $ on $ \Theta $. We call  $ \op{rot}_f(T,A) $ the fibered rotation number, which is independent of the choice of the lift. Moreover, if $ (\Theta,T,\mu) $ is further uniquely ergodic, the limit is uniform in all $ (\theta,W_{\Lambda})\in \Theta\times\mathbf{U}(m) $.

Let $ (T,A_E) $ be a continuous family of Schr\"odinger cocycles. Then the fibered rotation number $ \op{rot}_f(T,A_E) $ is continuous with respect to $ E $. The \textit{fibered rotation function} of this family, denoted by $ \op{rot}_f(E) $, is defined as a continuous branch of the fibered rotation number $ \op{rot}_f(T,A_E) $.
Then Theorem \ref{t:gablabel} is the corollary of the following result:
\begin{theorem}\label{Mainthm}
	Given $(\Theta,T,\mu)$. 
	Let $ \mathcal{N}(E) $ denotes the {\it integrated density of states} of the generalized Schr\"odinger operator $ H_{\theta} $, and $ \op{rot}_f(E) $ denotes the fibered rotation function of the associated Schr\"odinger cocycle family $ (T,A_E) $. Then we have
	\[
		m\pa{1-\mathcal{N}(E)}=\op{rot}_f(E) \mod \mu(C(\Theta,\Z)).
	\] 
    The modulo means different choice of fibered rotation function would differ a possible value of \( \mu(C(\Theta,\Z)) \).
\end{theorem}

\begin{remark}
	If the cocycle $ (T,A) $ is restricted in $ \rmm{Sp}(2m,\R) $, one can also consider its invariant action on $ \Theta\times \op{Lag}( \mathbb{R}^{2m},\omega_m ) $, and follows a same line, the fibered rotation number is well-defined, the value must coincide.
\end{remark}

\begin{remark}

	As a comparison, recall the previous result (see \cite{Kri2004Reducibility} for the quasiperiodic case as an example) for the one-dimensional Schrödinger operator, which states that
\[
1-\mathcal{N}(E) = 2\op{rot}_f(E) \mod \mathbb{Z},
\]  
which differs from our result. This discrepancy arises because the fibered rotation number we consider is defined for Hermitian symplectic cocycles, a framework that is more natural in our setting. 

There exists an intrinsic “doubling” relation between the symplectic and Hermitian symplectic actions. To illustrate this, note that in the case \( m=1 \), for any real Lagrangian subspace \( \Lambda = \begin{bsmallmatrix} x \\ y \end{bsmallmatrix} \), the previous definition used the following mapping between \( \op{Lag}(\mathbb{R}^{2},\omega_1) \) and \( \mathbf{U}(1) \):
\[
W_{\Lambda}^{\frac{1}{2}} := (x + iy) \left( \sqrt{x^2 + y^2} \right)^{-1} \in \mathbf{U}(1).
\]
One can easily verify that \( W_{\Lambda} = W_{\Lambda}^{\frac{1}{2}} \cdot \left( W_{\Lambda}^{\frac{1}{2}} \right)^T \). However, the map \( W_{\Lambda}^{\frac{1}{2}} \) is not well-defined. Indeed, given two different Lagrangian frames of the same Lagrangian subspace, say \( \begin{psmallmatrix} X \\ Y \end{psmallmatrix} \) and \( \begin{psmallmatrix} X \\ Y \end{psmallmatrix} R^{-1} \), the corresponding images under \( W_{\Lambda}^{\frac{1}{2}} \) may differ. 
\end{remark}
\begin{remark}
	On the other hand, we note that in the recently updated work of Avila–You–Zhou \cite{avila2016dry}, they identify \( \operatorname{Lag}(\mathbb{R}^{2},\omega_1) \) with \( \mathbb{R}\mathbb{P}^1 \), which essentially aligns with our perspective in defining the fibered rotation number. This insight also inspires a valuable spectral application \cite{LXZ}.
\end{remark}

\subsection{Mean Maslov index}
The fibered rotation number $ \op{rot}_f(T,A) $ is not a real number, but an equivalence class in $ \R$. It is closely related to the concept of {\it mean Maslov index}. In fact, under the assumption that $A(\cdot)$ is continuous and homotopic to the identity, there is a continuous path $\gamma_u$  from $ \op{Id}$ to $A(u)$.  
	Then the path $\Phi^t (\theta)=\cdots\gamma_{T^n \theta}\cdots\gamma_{T\theta} \gamma_\theta $
	satisfies $ \Phi^n(\theta)\Lambda=A^n(\theta)\Lambda $, where $ A^n(\theta) $ is the $ n $-step transfer matrix of $ A $. The determinant of $ W_{\Lambda}^{-1}W_{\Phi^t\Lambda} $ induces a continuous path on $ S^1 $, then the fibered rotation number representative  under such homotopy is exactly the mean increment along the circumference of a continuous determination of $ \op{arg}:S^1\to \R $, i.e.,
	\begin{equation} \label{representative}
		\lim_{t\to\infty}\frac{1}{2\pi}\frac{1}{t}\arg\det W_{\Lambda}^{-1} W_{\Phi^t(\theta)\Lambda}, \text{ for } \mu\text{-a.e.} \theta\in\Theta \text{ and any } \Lambda\in \op{Lag}( \mathbb{C}^{2m},\omega_m ). 
	\end{equation} 
	In particular, when restricted to the symplectic case, the map $ \det W_{\Lambda}  $ coincides with the  
	map $ \op{Det}^2:\op{Lag}( \mathbb{R}^{2m},\omega_m )\to\R $, $ \Lambda=\begin{bsmallmatrix}
		X\\Y
	\end{bsmallmatrix}\mapsto \frac{\det^2(X+iY)}{\det (X^TX+Y^TY)}  $ defined by Arnold \cite{Arnold-Maslov}.

Let us recall a  definition of Maslov index in \cite[Section 2.2]{BZ18}.
	There are $m$ continuous functions $\theta_j\in C([a,b],\mathbb{R})$ such that $e^{i\theta_j(t)},1\leq j\leq m,$ are all the eigenvalues of $ W_{\Lambda_0}^{-1}W_{\Lambda(t)}$ (counting algebraic multiplicities).
	Denote by $[a]$ the integer part of $a\in \mathbb{R}$.
	Define $E(a)=-[-a]$.
	Then the Maslov index can be defined as
	\begin{equation*}
		\mu(\Lambda_0,\Lambda(t),t\in[a,b])=\sum_{j=1}^{m}\left(E\left(\frac{\theta_j(b)}{2\pi}\right)-E\left(\frac{\theta_j(a)}{2\pi}\right)\right).
	\end{equation*}
	In fact, it is not the original definition of Maslov index.  This formula was proved by  Liu {\it et al.} in \cite{Liu2007},\cite{LWL2011} and \cite{Liu2019}  for Lagrangian boundary conditions, and then it was used in \cite{BZ18} as a definition. 
	It is clear that
	\[
	\left|\frac{1}{2\pi}\arg\det(W_{\Lambda(t)})-\sum_{j=1}^m E\left(\frac{\theta_j(t)}{2\pi}\right)\right| < M
	\]
	for some constant $M$.
	Then we have 
	\[
	\lim_{t\to\infty}\frac{1}{2\pi}\frac{1}{t}\arg\det W_{\Lambda}^{-1}W_{\Phi^t(\theta)\Lambda}=\lim_{t\to \infty} \frac{1}{t} \mu(\Lambda,\Lambda(s),s\in [0,t]).
	\]
	So the fibered rotation number is the mean Maslov index.
	Usually, mean Maslov index is defined only for periodic potential and is called mean index. It is a special case of mean Maslov index $\lim_{t\to \infty} \frac{1}{t} \mu(\Lambda,\Lambda(s),s\in [0,t]) $.

\subsection{Applications}

We consider perturbations of ergodic generalized Schr\"odinger operators in $ \ell^2(\Z,\C^m) $ by the addition of a random potential of diagonal type. Which can be viewed as a generalization of \cite{ADG2023Spectruma}.

Let us state our result precisely. The unperturbed model is given as \eqref{eq:H2}, equipped with an ergodic topological dynamical system $(\Theta,T,\mu)$ with $\supp\mu = \Theta$. We denote its almost sure spectrum by $ \Sigma_0 $. 

The random perturbation is given by

\[
	W_\omega(n)=\omega_n=\begin{pmatrix}
		\omega_{n,1}&&&\\
		&\omega_{n,2}&&\\
		&&\ddots&\\
		&&&\omega_{n,m}\\
	\end{pmatrix}, \ \omega\in\Omega, n\in\Z,
\] 
where $ \omega_{n,l} $, $ n\in\Z $, $ l=1,\dots,m $, are independent identically distributed bounded real random variables with probability distribution $ \nu $, satisfying $ \sharp \pa{ S:=\supp\nu}\geq 2 $ and
\[
	\Omega=\begin{pmatrix}
		S&&&\\
		&S&&\\
		&&\ddots&\\
		&&&S\\
	\end{pmatrix}^{\Z}.
\]  
Since the product of $ \mu $ and $ \tilde{\mu}:=(\diag{\nu,\nu,\dots,\nu})^{\Z} $ is ergodic with respect to the product of $ T $ and the left shift action on $ \Omega $. We denote its almost sure spectrum by $ \Sigma_1 $.
\begin{proposition}\label{prop:smallerspec}
	$ \sigma(H_\theta+W_{\omega})\subseteq \Sigma_1  $ for every $ (\theta,\omega)\in \Theta\times\Omega $.
\end{proposition}
\begin{proof}
	Since $ \supp \mu\times\tilde{\mu}=\Theta\times\Omega $, then the proof is a consequence of Theorem \ref{Haropuig} and is essentially contained in \cite[Theorem 4.9.3]{DF}.
\end{proof}

\begin{definition}
	Suppose $ A $ and $ B $ are compact subsets of $ \R $. We define the compact set $ A\bigstar B $ as follows. If $ \op{diam}(A) \geq \op{diam}(B) $, then $ A\bigstar B  = A + \op{ch}(B) $, and if $\op{diam}(A) < \op{diam}(B) $, then $ A\bigstar B = \op{ch}(A) + B $. Here, $ \op{diam}(S) $ denotes the diameter and $ \op{ch}(S) $ denotes the convex hull of a compact $ S \subset R $, and $ S_1 + S_2 $ denotes the Minkowski sum $ \{s_1 + s_2 : s_1 \in S_1, s_2 \in S_2\} $.
\end{definition}

\begin{theorem}\label{Thm:bigstar}
	Consider the setting described above and assume that $ \Theta $ is connected. Then, we have
	\[
		\Sigma_1=\Sigma_0\bigstar S.
 	\] 
\end{theorem}
\begin{corollary}
	If $ \Theta $ is connected, then the almost sure spectrum $ \Sigma_1 $ is given by a finite union of non-degenerate compact intervals.
\end{corollary}

\begin{example}
	As a direct example, the almost sure spectrum of the Anderson model on the strip introduced by \cite{KLS1990Localization}, is given by a finite union of non-degenerate compact intervals.
\end{example}

\ \\

The remainder of the paper is organized as follows. Some preliminaries are introduced in Section \ref{sec:pre}. The fibered rotation number for Hermitian symplectic cocycles is defined in Section \ref{sec:rotation_number}. Theorem \ref{Mainthm}, Theorem \ref{t:gablabel}, and Theorem \ref{Thm:bigstar} are proved in Sections \ref{sec:idseqrot}, \ref{sec:GLT}, and \ref{sec:interval}, respectively. Finally, we provide some useful almost invariant properties potentially used in the reducibility theory, as well as further details on Remark \ref{fractional-label} and Remark \ref{rem:Alter}.

\section{Preliminaries}\label{sec:pre}

Let $ \|\cdot\| $ be the spectral norm. Let $ \F $ denote $ \R $ or $ \C $. Let $ \op{Mat}_m \F $ denote the set of $ m\times m $ matrices over $ \F $. Let $ \op{Sym}_m \F $ denote the set of symmetric $ m\times m $ matrices over $ \F $. The symplectic group over $\F$, denoted by $\rmm{Sp}(2m,\F)$, is the group of all matrices $ M\in \rmm{GL}(2m,\F) $ satisfying
\[
	M^TJM=J, \text{with } J=
\begin{pmatrix}0&I_m\\-I_m&0\end{pmatrix}.
\] 
The Hermitian symplectic group $\rmm{HSp}(2m,\C)$ is defined as
\[
	\rmm{HSp}(2m,\C):=\{M\in \rmm{GL}(2m,\C): M^\ast J M=J\}.
\] 
The pseudo unitary group $\mathbf{U}(m,m)\subset \rmm{GL}(2m,\C)$ is defined as
\[
	\mathbf{U}(m,m):=\set{A \subset \rmm{GL}(2m,\C): A^\ast \begin{pmatrix}I_d&\\&-I_d\end{pmatrix}A=\begin{pmatrix}I_d&\\&-I_d\end{pmatrix}}.
\] 
\subsection{(Hermitian) symplectic group actions}\label{sec:symp}
Let 
 $ \op{Lag}( \mathbb{F}^{2m},\omega_m ) $ denote the Lagrangian Grassmannian of $ ( \mathbb{F}^{2m},\omega_m ) $.
The Lagrangian frame of a given $\Lambda\in \op{Lag}( \mathbb{F}^{2m},\omega_m )$  is defined by an injective linear map $\mathcal{L}:\mathbb{F}^{m}\rightarrow \Lambda$ with the form $\mathcal{L}=\begin{psmallmatrix}X\\Y \end{psmallmatrix}$, where $X$ and $Y$ are $m\times m$ matrices over $ \F $ such that $X^*Y=Y^*X$ and $\op{rank}(\mathcal{L})=m$.
 Here $X^*$ is  the conjugate transpose of  $X$.
Two frames are equivalent, $ \mathcal{L}_1\sim \mathcal{L}_2 $, if $ \mathcal{L}_1=\mathcal{L}_2 R $ for $ R\in \rmm{GL}(m,\F) $. Thus, we can identify $ \Lambda= \begin{bsmallmatrix}
	X\\Y
  \end{bsmallmatrix}$, if we use $ \begin{bsmallmatrix}
	X\\Y
  \end{bsmallmatrix} $ to represent the class of $ \begin{psmallmatrix}X\\Y \end{psmallmatrix}_{\sim} $.
The action of $ \rmm{Sp}(2m,\R) $, $ \rmm{HSp}(2m,\C) $ leaves $ \op{Lag}( \mathbb{R}^{2m},\omega_m ) $, $ \op{Lag}( \mathbb{C}^{2m},\omega_m ) $ invariant respectively. 

Consider the Cayley element 
\begin{equation}\label{Cayley ele}
\mathcal{C}:=\frac{1}{\sqrt{2i}}
\begin{pmatrix}
	I_m& -i\cdot I_m\\ I_m& i\cdot I_m
\end{pmatrix}.
\end{equation}
For any $2m\times 2m$ complex matrix $A$, we denote $\Aa:=\mathcal{C}A\mathcal{C}^{-1}$, then the map  $A\mapsto \Aa$ is a Lie group isomorphism from $\rmm{Sp}(2m,\R)$, $\rmm{HSp}(2m,\C)$ to $\mathbf{U}(m,m)\cap \rmm{Sp}(2m, \C)$, $\mathbf{U}(m,m)$ respectively. 
Denote \begin{eqnarray*}
	{\overset{\circ}{\op{Lag}}}(\mathbb{R}^{2m},\omega_m)&=&\set{\begin{pmatrix}
		M\\ I_m
	\end{pmatrix}_{\sim}: M\in \mathbf{U}(m)\cap \op{Sym}_m \C},\\
	{\overset{\circ}{\op{Lag}}}(\mathbb{C}^{2m},\omega_m)&=&\set{\begin{pmatrix}
		M\\ I_m
	\end{pmatrix}_{\sim}: M\in \mathbf{U}(m)}.
\end{eqnarray*}
Note that the Cayley element induces an isomorphism from $ \op{Lag}( \mathbb{F}^{2m},\omega_m ) $ to $ {\overset{\circ}{\op{Lag}}}( \mathbb{F}^{2m},\omega_m ) $,  $ \Lambda=\begin{bsmallmatrix}
	X\\Y
\end{bsmallmatrix}\mapsto \begin{bsmallmatrix}
	I_m \\ W_\Lambda
\end{bsmallmatrix}:={\overset{\circ}{\Lambda}} $, where $ W_\Lambda=(X+iY)(X-iY)^{-1} $. Clearly, the action of $\mathbf{U}(m,m)\cap \rmm{Sp}(2m, \C)$, $\mathbf{U}(m,m)$ leaves $ {\overset{\circ}{\op{Lag}}}(\mathbb{R}^{2m},\omega_m) $, $ {\overset{\circ}{\op{Lag}}}(\mathbb{C}^{2m},\omega_m) $ invariant respectively.
In summary, we have the following commutative diagram:
	$$\xymatrix{
		{\op{Lag}}(\mathbb{F}^{2m},\omega_m)\ar[rr]^{A}\ar[d]_{\mathcal{C}} & & {\op{Lag}}(\mathbb{F}^{2m},\omega_m) \ar[d]^{\mathcal{C}} \\
		{\overset{\circ}{\op{Lag}}}(\mathbb{F}^{2m},\omega_m)\ar[rr]^{\Aa} & & {\overset{\circ}{\op{Lag}}}(\mathbb{F}^{2m},\omega_m)
	}$$

\subsection{The integrated density of states} Let $ \{H_\theta\}_{\theta\in\Theta} $ be a topological family of generalized Schr\"odinger operator acting on $ \ell^2(\Z,\C^m) $ defined over an ergodic topological dynamical system $(\Theta,T,\mu)$. 
	The density of states measure (DOSM) is the measure $ k=k_{\mu} $ defined by 
	\[
		\int g dk=\int_{\Theta} \frac{1}{m}\sum_{i=1}^m \langle \delta_{0,i},g(H_{\theta})\delta_{0,i}\rangle d\mu(\theta)
	\] 
	for bounded measurable $ g $.
The integrated density of states (IDS) is then defined by 
\begin{equation} 
	\mathcal{N}(E)=\int \chi_{(-\infty,E]} dk.
\end{equation} 

For any $N \in \mathbb{N}_+$, let $H_{\theta}^{N}$ be the restriction of $H_{\theta}$ to the finite-dimensional space $\ell^{2}([1,N];\mathbb{C}^m)$, given by the block tridiagonal matrix
\begin{equation} 
	H_{\theta}^N=\begin{pmatrix}
		B_{\theta}(N)& C&\\
		C & B_{\theta}(N-1)& C^*\\
		& \ddots & \ddots &\ddots & \\ 
		& & C & B_{\theta}(2)& C^*\\
		& & &  C & B_{\theta}(1) \\
	
		\end{pmatrix}\in \rmm{Mat}_{mN}\C.
\end{equation} 
Then, following a standard proof (see \cite{DF}), one can show that  the density of states measure $ k $ could also be given by 
\begin{eqnarray*}
	\int g d k=\lim_{N\to\infty}\int g d k_\theta^N:=\lim_{N\to\infty}\frac{1}{mN}\op{Tr}(g(H_{\theta}^N)),& \text{ for }\mu\text{-a.e.\ }\theta \in \Theta.
\end{eqnarray*}
As an immediate corollary, let $\mathcal{N}_{\theta}^{N}: \mathbb{R} \rightarrow \mathbb{Z}_{+}$ be given by the law
	\begin{equation*}
	\begin{array}{lllll}
		\mathcal{N}_{\theta}^{N}(E) &  := & \dfrac{1}{mN} \sharp\set{ \text{ eigenvalues of }H_{\theta}^{N}\text{ least or equal to }E }.
	\end{array}
	\end{equation*}
	It  follows that
	\begin{eqnarray} 
		\lim\limits_{N\to\infty} \mathcal{N}_{\theta}^{N}(E)=\mathcal{N}(E), &\text{ for }\mu\text{-a.e.\ }\theta \in \Theta.
	\end{eqnarray}
	Clearly, $ \mathcal{N} $ is continuous in $ E $.
\subsection{Transfer matrices}
Any formal solution $ \textbf{u}: \Z\to \C^m $ solves $ H_{\theta}\textbf{u}=z\textbf{u} $ if and only if for every $ n\in\Z $,
\begin{equation} 
    \begin{pmatrix}
        \textbf{u}_{n+1}\\
        \textbf{u}_n
    \end{pmatrix}=\widehat{A}_z(T^n\theta)\begin{pmatrix}
        \textbf{u}_{n}\\
        \textbf{u}_{n-1}
    \end{pmatrix}
\end{equation} 
where $ \widehat{A}_z(\cdot)=\begin{psmallmatrix}
	C^{-1}(z-f(\cdot ))&-C^{-1}C^*\\
	I_m&O_m
\end{psmallmatrix} $.
We denote its $ n $-step transfer matrix by
\begin{equation} 
	\widehat{A}_{z,n}(\theta)= \begin{cases}
		\widehat{A}_z(T^{n-1}\theta)\widehat{A}_z(T^{n-2}\theta)\cdots \widehat{A}_z(\theta) & n\geq 1,\\
		I_{2m}& n=0,\\
		\widehat{A}_z(T^{n}\theta)^{-1}\widehat{A}_z(T^{n+1}\theta)^{-1}\cdots \widehat{A}_z(T^{-1}\theta)^{-1}& n\leq -1.
	\end{cases}
\end{equation} 
Then we have
\begin{equation} 
	\begin{pmatrix}
		\textbf{u}_{n+1}\\
		\textbf{u}_n
	\end{pmatrix}=
	\widehat{A}_{z,n}(\theta)\begin{pmatrix}
		\textbf{u}_1\\
		\textbf{u}_0
	\end{pmatrix}.
\end{equation} 
In particular, for each $ z\in \C $, let $ U_z,\,V_z: \Z\to \rmm{Mat}_m\C $ be such that
	\begin{equation} 
		\begin{pmatrix}
			U_z(n+1)&V_z(n+1)\\
			U_z(n)&V_z(n)
		\end{pmatrix}=\widehat{A}_{z,n}(\theta)\begin{pmatrix}
			I_m& 0\\
			0&I_m
		\end{pmatrix}.
	\end{equation}
	Accordingly, the $ n $-step transfer matrix of $ A_{z} $, denoted by $ A_{z,n} $, can be calculated as
	\begin{equation} 
		A_{z,n}(\cdot):=P\widehat{A}_{z,n}(\cdot)P^{-1}
			=\begin{pmatrix}
				CU_z(n+1)C^{-1}&CV_z(n+1)\\
				U_z(n)C^{-1}&V_z(n)
			\end{pmatrix}.
	\end{equation} 
	Denote the iterates of the corresponding cocycles $ (T, \widehat{A}_z) $ and $ (T, A_z) $ by 
	\begin{eqnarray*}
		(T, \widehat{A}_z)^n  = (T^n, \widehat{A}_{z,n}), &
		(T, A_z)^n  = (T^n, A_{z,n})
	\end{eqnarray*} 
	respectively.

\subsection{Uniformly hyperbolicity}
We say that a skew-product $(T,\Phi) $ in $ X\times \C^{2m} $ is {\it{uniformly hyperbolic}} if there exist constants $c > 0, L > 1$ and an invariant Whitney splitting $ X\times \C^{2m} =\Lambda^s\oplus \Lambda^u $ 
 such that the following statements hold:
\begin{itemize}
\item[{\rm(a)}] (Invariance) For all $x \in X$, one has
\begin{equation}
\Phi(x) \Lambda^s (x) = \Lambda^s(Tx),
\quad
\Phi(x)\Lambda^u (x) = \Lambda^u(T x).
\end{equation}

\item[{\rm(b)}] (Contraction) For all $n > 0$, $x\in X$, $v_s \in \Lambda^s(x)$, and $ v_u \in \Lambda^u(x) $, one has
\begin{equation}
\| \Phi^n   (x) v_s \| \leq c L^{-n} \|v_s\|,
\quad
\| \Phi^{-n}(x) v_u \| \leq c L^{-n} \|v_u\|.
\end{equation}
\end{itemize}




\begin{theorem}[\cite{HP2013Thouless}]\label{Haropuig}
	Suppose that
	\begin{itemize}
		\item[\rm{(HS1)}] The base dynamical system $ (\Theta,T,\mu) $ is an ergodic topological dynamical system;
		\item[\rm{(HS2)}] The orbit of the base point $ \theta_0\in\Theta $ is  dense in $ \Theta $. 
	\end{itemize}
	Then the spectrum of $ H_{\theta_0} $ acting on $ \ell^2(\Z,\C^m) $ is the set $ E\in\C $ for which $ (T,A_E) $ is not uniformly hyperbolic.
\end{theorem}

\section{The fibered rotation number}\label{sec:rotation_number}

Let $ (\Theta,T,\mu) $ be given. Assume that $ A(\cdot):\Theta \to \rmm{HSp}(2m,\C) $ is continuous and {\it{homotopic to the identity}}, that is, there exists a homotopy $ H:[0,1]\times \Theta \to \rmm{HSp}(2m,\C) $ such that $ H(0,\cdot)=A(\cdot)$,  $H(1,\cdot)= I_m $. We associate
to a {\it{cocycle}} $ (T,A) $ given by $ \Theta\times \op{Lag}( \mathbb{C}^{2m},\omega_m ) \circlearrowleft : (\theta,\Lambda)\mapsto(T\theta,A(\theta)\Lambda) $. Then $ (T,A) $ is {\it{homotopic to the identity}}, in the sense of there exists a homotopy $ H:[0,1]\times \op{Lag}( \mathbb{C}^{2m},\omega_m ) \to \op{Lag}( \mathbb{C}^{2m},\omega_m ) $ such that $ H(0,\cdot)=A(\cdot)$,  $ H(1,\cdot)=\op{id}(\cdot) $.
The Cayley element actually induces a map homotopic to the identity $ \Theta\times \mathbf{U}(m)\circlearrowleft: (\theta, W_{\Lambda})\mapsto(T\theta, W_{A(\theta)\Lambda}) $, where we still denote by $ (T,A) $, and unless otherwise noted, we will always refer to this cocycle. 

For $ m =1 $, the unitary group $ \mathbf{U}(1) \cong S^1 $, and in this case, there is an obvious way to define the ``rotation number''. But in higher dimensions, the rotation number comes more indirectly.

Let $ \varphi: S^1 \times \mathbf{SU}(m) \to \mathbf{U}(m)  $ be given by
\[
	(e^{2\pi i \phi},S)\mapsto e^{2\pi i \phi}S.
\] 
This is a surjective Lie group homomorphism with kernel $ \{(e^{-2\pi i \frac{j}{m}},e^{2\pi i \frac{j}{m}}I_m)\}_{j=0}^{m-1} \cong \Z_m $, hence there is a Lie group isomorphism $ \frac{S^1 \times \mathbf{SU}(m)}{\ker \varphi}\cong \mathbf{U}(m) $; \cite{universalcover}. Moreover, $ S^1 \times \mathbf{SU}(m)$ is  an $ m $-fold covering of $  \mathbf{U}(m)  $.
Indeed, 
for any $ B\in \mathbf{U}(m) $, let $ \{e^{2\pi i\varphi_j}\}_{j=1}^m $ with all $ \varphi_j\in[0,1) $ be its eigenvalues (counted with multiplicities), then $ \det B=e^{2\pi i \pa{\sum_{j=1}^m{\varphi_j}\mod 1}} $. Let
\begin{equation} \label{rep_B}
	\phi_B=\frac{1}{m}\sum_{j=1}^m{\varphi_j}\in [0,1),\ \  S_B=e^{-2\pi i\phi_B}B\in \mathbf{SU}(m).
\end{equation} 
Then $ B=e^{2\pi i \phi_B}S_B $ can be uniquely represented as 

\[
	([e^{2\pi i \phi_B}],[S_B]):=\pa{e^{2\pi i \pa{\phi_B\mod 1/m}},e^{-2\pi i \pa{\phi_B\mod 1/m}}e^{2\pi i \phi_B}S_B }\in \frac{S^1 \times \mathbf{SU}(m)}{\ker \varphi} .
\] 
Generally, $ e^{2\pi i \phi}S $ can be uniquely represented as 
\[
	([e^{2\pi i \phi}],[S]):=\pa{e^{2\pi i \pa{\phi \mod 1/m}},e^{-2\pi i \pa{\phi \mod 1/m}}e^{2\pi i \phi}S}.
\] 
Then rewrite any $ (e^{2\pi i \phi},S)\in S^1 \times \mathbf{SU}(m)  $ as $ (e^{2\pi i \phi},e^{-2\pi i \phi}e^{2\pi i \phi}S) $, one can see that $ \varphi $ can be represented as $ \varphi: (e^{2\pi i \phi},S)\mapsto ([e^{2\pi i \phi}],[S]) $. One can check that this is an $ m $-fold covering map.

Let us denote by 
$ \pi:\R\mapsto S^1 $ the projection $ \pi(x)=e^{2\pi i x} $.

\begin{proposition}
	There exist continuous lifts 
	\[
		\begin{aligned}
			\widehat{(T,A)}: \Theta \times S^1 \times \mathbf{SU}(m) &\to \Theta \times S^1 \times \mathbf{SU}(m),\\
		\tilde{F}_{\,T,A}: \Theta \times \R \times \mathbf{SU}(m) &\to \Theta \times \R \times \mathbf{SU}(m)
		\end{aligned}
	\]   such that the following diagram is commutative:
	\[
	\begin{tikzcd}
	\Theta \times \R \times \mathbf{SU}(m) \arrow{r}{\tilde{F}_{\,T,A}} \arrow[swap]{d}{\tilde{\pi}:=\op{id}\times \pi \times \op{id}} & \Theta \times \R \times \mathbf{SU}(m)\arrow{d}{\tilde{\pi}} \\
	\Theta \times S^1 \times \mathbf{SU}(m) \arrow{r}{\widehat{(T,A)}} 
	\arrow[swap]{d}{\tilde{\varphi}:=\op{id}\times \varphi} & \Theta \times S^1 \times \mathbf{SU}(m) \arrow{d}{\tilde{\varphi}}\\
	\Theta \times \mathbf{U}(m)\arrow[swap]{d}{\cong } \arrow{r}{(T,A)} &  \Theta \times \mathbf{U}(m)\arrow{d}{\cong }\\
	\Theta \times \frac{S^1 \times \mathbf{SU}(m)}{\ker \varphi} & \Theta \times \frac{S^1 \times \mathbf{SU}(m)}{\ker \varphi} 
	\end{tikzcd}
\]
\end{proposition}
\begin{proof}
Identify $ \frac{S^1 \times \mathbf{SU}(m)}{\ker \varphi} $ with $ \mathbf{U}(m) $ for simplicity. 
Then we can see
\[
	\begin{aligned}
		\tilde{\pi}&: (\theta,x,S)\mapsto (\theta,e^{2\pi i x},S),\\
		\tilde{\varphi}&: (\theta,e^{2\pi i x},S)\mapsto (\theta,[e^{2\pi i x}],[S]),\\
		(T,A)&:(\theta,[e^{2\pi i x}],[S])\mapsto (T\theta,[e^{2\pi i x_{A(\theta)}}],[S_{A(\theta)}]),
	\end{aligned}
\] 
where $ ([e^{2\pi i x_{A(\theta)}}],[S_{A(\theta)}])=W_{A(\theta)\Lambda} $, if we identify $ \pa{[e^{2\pi i x}],[S]}=W_{\Lambda} $. 

Recall some basic facts about the covering space:
\begin{lemma}
	\label{Munkres1}
	If $ p : E \to B $ and $ p': E' \to B' $ are covering maps, then
$ p \times p'  : E \times E' \to B \times B'$
is a covering map.
\end{lemma}
\begin{lemma}
	\label{Munkres2}
   Let $ q : X \to Y $ and $ r : Y \to Z $ be covering maps; let $ p = r \circ q $. If
$ r^{-1}(z) $ is finite for each $ z \in Z $, then $ p $ is a covering map.
\end{lemma}
\begin{lemma}[homotopy lifting property]\label{Hatcher}
	Given a covering space $ p:\tilde{X}\to X $, a homotopy $ f_t:Y\to X $, and a map $ \tilde{f}_0: Y\to \tilde{X} $ lifting $ f_0 $, then there exists a unique homotopy $ \tilde{f}_t: Y\to \tilde{X} $ of $ \tilde{f}_0 $ that lifts $ f_t $.
\end{lemma}
By Lemma \ref{Munkres1}, and Lemma \ref{Munkres2}, we have both $ \tilde{\varphi} $ and  $ \tilde{\varphi}\circ\tilde{\pi}
 $ are covering maps. 
 Consider $ (T,A)\circ \tilde{\varphi} $. Since $ (T,A) $ is homotopic to the identity, there exists a homotopy $ f_t: \Theta \times S^1 \times \mathbf{SU}(m) \to \Theta \times \mathbf{U}(m) $, with $ f_0=(T,\op{id})\circ\tilde{\varphi} $, $ f_1=(T,A)\circ \tilde{\varphi} $. 
 Note that $ (T,\op{id}):\Theta \times S^1 \times \mathbf{SU}(m)\circlearrowleft $ lifting $ f_0 $, by Lemma \ref{Hatcher},  there exists a homotopy $ \tilde{f}_t: \Theta \times S^1 \times \mathbf{SU}(m)\circlearrowleft $, with $ \tilde{f}_0=(T,\op{id}) $, and $ \tilde{f}_1=\widehat{(T,A)} $ is the lift of $ (T,A)\circ \tilde{\varphi} $, i.e., $ (T,A)\circ \tilde{\varphi}=\tilde{\varphi}\circ\widehat{(T,A)} $. 
 Follows a same line, consider $ \widehat{(T,A)}\circ\tilde{\pi} $, there exists a map $ \tilde{F}_{\,T,A}:\Theta \times \R \times \mathbf{SU}(m) \circlearrowleft $ homotopic to the identity and is the lift of $ \widehat{(T,A)}\circ\tilde{\pi} $, i.e., $ \widehat{(T,A)}\circ\tilde{\pi}=\tilde{\pi}\circ \tilde{F}_{\,T,A} $. Then we have $ (T,A)\circ \tilde{\varphi}\circ\tilde{\pi}=\tilde{\varphi}\circ\tilde{\pi}\circ \tilde{F}_{\,T,A} $, i.e., $ \tilde{F}_{\,T,A} $ is a lift of $ (T,A)\circ \tilde{\varphi}\circ\tilde{\pi} $.
\end{proof}

\begin{proposition}\label{liftprop}
	Let $ p_j $, $ j=1,2,3 $ be the canonical projections on the corresponding factor. Viewing $ \tilde{F}_{\,T,A} $ as the lift  of $ (T,A) $, we have the following properties:
\begin{enumerate}
	\item[{\rm (1)}] Let $ \tilde{F}_{\,T,A},\tilde{F}_{\,T,A}' $ be two lifts of $ (T,A) $, we have 
	\[
		p_2\circ\tilde{F}_{\,T,A}(\theta,x,S)- p_2\circ\tilde{F}_{\,T,A}'(\theta,x,S) \in \frac{\Z}{m}.
	\] 
	In particular, if $ \tilde{F}_{\,T,A} $ is a lift of $ (T,A) $, for any continuous integer-valued function $ \ell(\theta,x,S):\Theta\times  S^1\times\mathbf{SU}(m)\to\Z $, the map 
	\[
		\tilde{F}_{\,T,A}'(\theta,x,S)=(T\theta,p_2\circ \tilde{F}_{\,T,A}(\theta,x,S)+\frac{\ell(\theta,x,S) }{m},p_3\circ \tilde{F}_{\,T,A}(\theta,x,S))
	\] 
	is also a lift of $ (T,A) $.
	\item[{\rm (2)}] If $ \tilde{\varphi}\circ\tilde{\pi}(x,S)=\tilde{\varphi}\circ\tilde{\pi}(x',S') $, then we have
	\[
		p_2\circ\tilde{F}_{\,T,A}(\theta,x',S')-p_2\circ\tilde{F}_{\,T,A}(\theta,x,S)=x'-x \in \frac{\Z}{m}.
	\] 
\end{enumerate}
\end{proposition}
\begin{proof}
	(1) Since $ (T,A)\circ \tilde{\varphi}\circ\tilde{\pi}(\theta,x,S)=\tilde{\varphi}\circ\tilde{\pi}\circ\tilde{F}_{\,T,A}(\theta,x,S)=\tilde{\varphi}\circ\tilde{\pi}\circ\tilde{F}_{\,T,A}'(\theta,x,S) $, then we have $ [e^{2\pi i\, p_2\circ\tilde{F}_{\,T,A}(\theta,x,S)}]=[e^{2\pi i \, p_2\circ\tilde{F}_{\,T,A}'(\theta,x,S)}] $.
	This means $ p_2\circ\tilde{F}_{\,T,A}(\theta,x,S)-p_2\circ\tilde{F}_{\,T,A}'(\theta,x,S)=0 \mod \tfrac{1}{m} $. The second claim is easy to verify.

	(2) By (1), it suffices to consider the lift $ \tilde{F}_{\,T,A} $ induced by the homotopy $ f_t $ with $ f_0=(T,\op{id}) $, $ f_1={(T,A)} $. Then there exists a homotopy with $ \tilde{\tilde{f}}_0=(T,\op{id})  $, $ \tilde{\tilde{f}}_1=\tilde{F}_{\,T,A} $, such that  $ f_t\circ\tilde{\varphi}\circ\tilde{\pi}=\tilde{\varphi}\circ\tilde{\pi}\circ \tilde{\tilde{f}}_t $.  This forces $ p_2\circ\tilde{\tilde{f}}_t(\theta,x',S')-p_2\circ\tilde{\tilde{f}}_t(\theta,x,S)\equiv x'-x\in\frac{\Z}{m} $. \qedhere

\end{proof}

By (2) of Proposition \ref{liftprop}, the map 
\begin{eqnarray*}
	f :&\Theta\times \mathbf{U}(m) &\to \R\\
	&(\theta,[e^{2\pi i x}],[S])&\mapsto m\cdot\pa{ p_2\circ \tilde{F}_{\,T,A}(\theta,x,S)-x }
\end{eqnarray*}
is therefore well-defined.
 Notice that the map $ f $ is independent of the choice of the lift, up to the addition of a continuous integer-valued function $ \ell(\theta,[e^{2\pi i x}],[S]):\Theta\times\mathbf{U}(m)\to\Z $. Then we define 
\begin{equation} \label{rot_f}
	\begin{aligned}
		\op{rot}_f(T,A)(\theta,[e^{2\pi i x}],[S])=&\lim_{n\to +\infty}\frac{1}{n}\sum_{k=0}^{n-1} f(\tilde{\varphi}\circ\tilde{\pi}\circ\tilde{F}_{\,T,A}^k(\theta,x,S))  \\
		=&\lim_{n\to +\infty}\frac{1}{n}\sum_{k=0}^{n-1} f\pa{(T,A)^k(\theta,[e^{2\pi i x}],[S])} 
	\end{aligned}
\end{equation} 
if the limit exists.

\begin{lemma}\label{independent}
	If the limit \eqref{rot_f} exists for some $ (\theta,[e^{2\pi i x}],[S]) $, then it coincides for any $ (\theta,[e^{2\pi i x'}],[S'])\in \{\theta\}\times  \mathbf{U}(m)  $. 
\end{lemma}
\begin{proof}[Proof of Lemma \ref{independent}]
	Without loss of generality, let $ \tilde{F}_{\,T,A} $ be the lift of $ {(T,A)} $ induced by a homotopy $ (T,A^t)\circ\tilde{\varphi}\circ\tilde{\pi} $, $ A^0=\op{id} $, $ A^1=A $. Fix $ \theta\in\Theta $. Let $ g=\varphi\circ\pa{\pi\times \op{id}} $.
	Take any $ \pa{x,S},\pa{x',S'}\in \R\times \mathbf{SU}(m) $, suppose $ g\pa{x^{},S^{}}=  W_{\Lambda} $,   $
	g\pa{x',S'}=W_{\Lambda'}$. For any $ n\geq 1 $, denote  $ \tilde{F}_{\,T,A}^n\pa{\theta, x,S}= \pa{T^n\theta, x_{A_{n}(\theta)}^{},S_{A_n(\theta)}^{}} $, $ \tilde{F}_{\,T,A}^n\pa{\theta, x',S'}= \pa{T^n\theta, x_{A_{n}(\theta)}',S_{A_n(\theta)}'} $, then we have $ g\pa{x_{A_{n}(\theta)}^{},S_{A_n(\theta)}^{}}=  W_{A_n(\theta)\Lambda}$,  $
	g\pa{x_{A_{n}(\theta)}',S_{A_n(\theta)}'}= W_{A_n(\theta)\Lambda'} $ and

\begin{equation} \label{eq:Deviation}
	\begin{aligned}
		&\av{\sum_{k=0}^{n-1}f(\tilde{\varphi}\circ\tilde{\pi}\circ\tilde{F}_{\,T,A}^k(\theta,x,S)) -\sum_{k=0}^{n-1}f(\tilde{\varphi}\circ\tilde{\pi}\circ\tilde{F}_{\,T,A}^k(\theta,x',S'))}\\
			&\qquad\qquad\qquad\qquad\qquad=\av{mx_{A_n(\theta)}^{}-mx_{A_n(\theta)}'-(mx-mx')
			}.
	\end{aligned}
\end{equation}

	In order to demonstrate the independence of $ ([e^{2\pi i x}],[S])\in \mathbf{U}(m) $ in \eqref{rot_f}, it is enough to prove that the right-hand side of \eqref{eq:Deviation} is uniformly bounded. It is important to note that the homotopy provides a path in $ \mathbf{U}(m) $: $ W_{A^t(\theta)\Lambda} $, where $ W_{A^0(\theta)\Lambda}=W_{\Lambda} $ and $ W_{A^n(\theta)\Lambda}=W_{A_n(\theta)\Lambda} $. This path determines a unique lifting path in $ \R\times\mathbf{SU}(m) $: $ (x^t,S^t) $, where $ x^0=x $ and $ x^n=x_{A_n(\theta)} $.
	Since unitary matrix is diagonalizable, and all its eigenvalues lie in $ \partial\D $, then according to \cite[Theorem II.5.2]{Kato1995Perturbation}, there exist continuous functions $ \rho_1(t),\dots,\rho_m(t): \R_+\to \partial\D $ such that the set of eigenvalues of $ W_{A^t(\theta)\Lambda} $, counted with their multiplicities, coincides with the unordered $ m $-tuple $ \{\rho_1(t),\dots,\rho_m(t)\} $. Therefore, there exist continuous argument functions $ \varphi_1(t),\dots,\varphi_m(t):\R_+\to\R $ such that $ \rho_j(t)=e^{2\pi i\, \varphi_j(t)} $ for $ j=1,\dots,m $. Any other selection of argument functions would differ by an integer. Note that $ e^{2\pi i\, m x^t}=\det W_{A^t(\theta)\Lambda} $, it is possible to select suitable argument functions such that $ m x^t=\sum_{j=1}^{m}\varphi_j(t) $.

	The following decomposition can be viewed as the singular value decomposition of $ \mathbf{U}(m,m) $, as described in \cite{Sad2015Herman}: for any $ \Aa\in\mathbf{U}(m,m) $, there exist $ U_1, U_2, V_1, V_2\in \mathbf{U}(m) $, and $ \Gamma $ is positive and diagonal such that
\[
	\Aa=\begin{pmatrix}
		U_1&\\ & U_2
	\end{pmatrix}\begin{pmatrix}
		\cosh(\Gamma)&\sinh(\Gamma)\\
		\sinh(\Gamma)&\cosh(\Gamma)
	\end{pmatrix}\begin{pmatrix}
		V_1 &\\ & V_2
	\end{pmatrix}=\begin{pmatrix}
		A_1& A_2\\A_3&A_4
	\end{pmatrix}.
\] 
Let $ \begin{psmallmatrix}
	X\\Y
\end{psmallmatrix} $ be any Lagrangian frame of $ \Lambda $, $ A\begin{psmallmatrix}
	X\\Y
\end{psmallmatrix}=\begin{psmallmatrix}
	X_1\\ Y_1
\end{psmallmatrix} $. Since 
\[
	\mathcal{C}A\begin{pmatrix}
		X\\Y
	\end{pmatrix}=\Aa \mathcal{C}\begin{pmatrix}
		X\\Y
	\end{pmatrix}=\Aa\begin{pmatrix}
		I_m\\ W_{\Lambda}
	\end{pmatrix}\begin{pmatrix}
		X-iY
	\end{pmatrix}=\begin{pmatrix}
		I_m\\ W_{A\Lambda}
	\end{pmatrix}\begin{pmatrix}
		X_1-iY_1
	\end{pmatrix},
\] 
this implies 
\begin{equation}\label{eq:W1234} 
	W_{A\Lambda}=(A_3+A_4 W_{\Lambda})(A_1+A_2 W_{\Lambda})^{-1}.
\end{equation}  Then performing a direct calculation, we obtain
\[
	\begin{aligned}
		W_{\Lambda}^{-1}W_{A\Lambda}^{}&=W_{\Lambda}^{-1}[U_2 \cosh(\Gamma) V_2]\cdot W_{\Lambda} \\
		&\qquad  (1+W_{\Lambda}^{-1}V_2^{-1}\tanh(\Gamma)V_1 )(1+V_1^{-1}\tanh(\Gamma)V_2 W_{\Lambda})^{-1}\\
		&\quad\  [U_1\cosh(\Gamma)V_1]^{-1}.
	\end{aligned}
\] 
Since $ \|W_{\Lambda}^{-1}V_2^{-1}\tanh(\Gamma)V_1\|, \|V_1^{-1}\tanh(\Gamma)V_2 W_{\Lambda}\| < 1 $, for any $ W_{\Lambda}\in \mathbf{U}(m) $. The spectrum of the matrix $ (1+W_{\Lambda}^{-1}V_2^{-1}\tanh(\Gamma)V_1 ),(1+V_1^{-1}\tanh(\Gamma)V_2 W_{\Lambda}) $ is always contained in a half-plane. Consider the paths $ W_{\Lambda}^{-1}W_{A^t(\theta)\Lambda} $, $ W_{\Lambda'}^{-1}W_{A^t(\theta)\Lambda'} $, it then follows that for any $ t\in\R_+ $, $ |m x^t-m x -(m {x'}^{t}-mx')|<m $. This completes the proof. \qedhere

\end{proof}

Let $ \nu $ be any probability measure which is invariant under $ (T,A) $ and which projects to $ \mu $ on $ \Theta $.
Then by Birkhoff ergodic theorem, the limit  \eqref{rot_f} exists for $ \nu $-almost every $ (\theta,W_{\Lambda})\in\Theta\times\mathbf{U}(m) $. More precisely, there is a set $ B_0\subset \Theta\times\mathbf{U}(m) $ with $ \nu \pa{\Theta\times\mathbf{U}(m)\setminus B_0}=0 $ such that $ \op{rot}_f(T,A)(\theta,W_{\Lambda}) $ exists for all $ (\theta,W_{\Lambda})\in B_0 $. Moreover, $ \op{rot}_f(T,A) $ is $ \nu $-integrable, 
\[
	\int_{\Theta\times\mathbf{U}(m) }\op{rot}_f(T,A)d\nu = \int_{\Theta\times\mathbf{U}(m) } f d\nu,
\] 
and is invariant under $ (T,A) $. 

By Lemma \ref{independent}, $ B_0 $ is of the form $ B_0=\Theta_0\times \mathbf{U}(m)  $, and $ \op{rot}_f(T,A) $ is independent of $ W_{\Lambda}\in \mathbf{U}(m) $. Thus $ \op{rot}_f(T,A) $ can be considered as a function on $ \Theta $, which is invariant under $ T $. Since $ \nu $ projects to $ \mu $ on $ \Theta $ and $ \mu $ is ergodic, we have 
\[
	\begin{aligned}
		\int_{\Theta\times\mathbf{U}(m) } f d\nu
		=&\int_{\Theta }\op{rot}_f(T,A)(\theta)d\mu(\theta)= \varrho.
	\end{aligned}
\] 
That is, there exists a set $ B_1=\Theta_1\times \mathbf{U}(m) $ with $ \nu(\Theta\setminus\Theta_1)=0 $, such that $ \op{rot}_f(T,A)$ equals $\varrho $ on $ B_1 $. Notice that the value $ \varrho $  does not depend on the choice of $ \nu $. 

Recall that the map \( f \) is independent of the choice of the lift, up to the addition of a continuous integer-valued function \( \ell
(\theta,x,S):\Theta\times\mathbf{U}(m)\to\mathbb{Z} \). Moreover, since \( \mathbf{U}(m) \) is connected, it follows that it is independent up to the addition of a continuous integer-valued function \( p(\theta)\in C(\Theta,\Z) \). Consequently, we deduce that \( \op{rot}_f(T,A) \) is independent of the choice of the lift, provided we take the equivalence class modulo all possible values of \( \mu(C(\Theta,\Z)) \).

\begin{remark}
	If $ (\Theta,T,\mu) $ is further uniquely ergodic, we can invoke a theorem by M.R. Herman \cite{herman1983methode} and Johnson-Moser \cite{gaplabel} to show that the limit in \eqref{rot_f} exists and coincides for  all $ \theta\in\Theta $. 
\end{remark}
\begin{remark}
	By the continuity of $ f $ with respect to $ A $, we immediately get  the continuity of $ \op{rot}_f(T,A) $  with respect to $ A $.
\end{remark}

\section{Proof of Theorem \ref{Mainthm}}\label{sec:idseqrot}

Let $ \Lambda_N(E) $ be $ A_{E,N}(\theta)\begin{bsmallmatrix}
	I_m\\ 0
\end{bsmallmatrix}=\begin{bsmallmatrix}
	CU_E(N+1)C^{-1}\\
	U_E(N)C^{-1}
\end{bsmallmatrix}  $. 
Notice that
\[
	\begin{aligned}
		A_E(\theta)=\begin{pmatrix}
			I&CC^*\\
			0&I
		\end{pmatrix}\begin{pmatrix}
			I&(E-B(\theta)-CC^*-I)\\
			0&I
		\end{pmatrix}\begin{pmatrix}
			C^{-1}&0\\
			C^{-1}&C^*
		\end{pmatrix}\begin{pmatrix}
			I&-CC^*\\
			0&I
		\end{pmatrix}.
	\end{aligned}
\] 
Construct a symplectic path $ P_{E}^t(\theta) $ satisfies $ P_{E}^n(\theta)=A_{E,n}(\theta) $, $ P_{E}^0(\theta)=I_m $ as follows:
Let $V(t)\in \op{GL}(m,\C) $ be a path connecting $C$ and $I_m$ with $V(1)=C$ and $ V(0)=I_m $, and 
	\begin{equation}\label{eq:def_Gi} G_1=\begin{psmallmatrix}
		I&-CC^*\\
		0&I
	\end{psmallmatrix},\   G_2(t)=\begin{psmallmatrix}
		I&t (E-B(\theta)-CC^*-I)\\
		0&I
	\end{psmallmatrix}\begin{psmallmatrix}
		I & \\
		tI & I
	\end{psmallmatrix}
	\begin{psmallmatrix}
		V(t)^{-1}& \\
		&V(t)^*
	\end{psmallmatrix},
\end{equation}
then set $ P_{E}^t(\theta)=G_1^{-1}G_2(t)G_1 $, for $ t\in[0,1] $. For any $ t\in [n,n+1) $, $ n\in \Z_+ $, extend $ P_{E}^t(\theta) = P_{E}^{t-n}(T^n\theta)A_{E,n}(\theta) $.
	This path provides a homotopy from $ (T,\op{id}) $ to $ (T,A_E) $, let $ \tilde{F}_{\,T,A_E} $ be the {\it unique}  lift of $ {(T,A_E)} $ induced by such a homotopy. 
	
	Since $ g (0,I_m)=W_{\Lambda_0} $. For any $ n\geq 1 $, denote $ \tilde{F}_{T,A_E}^n\pa{\theta, 0,I_m}= \pa{T^n\theta, x_{A_{E,n}(\theta)},S_{A_{E,n}(\theta)}} $.
	Following the same line as the discussion in the proofs of Lemma \ref{independent}, 
	the homotopy also provides a path in $ \mathbf{U}(m) $: $ W_{P_E^t(\theta)\Lambda_0} $, where $ W_{P_E^0(\theta)\Lambda}=W_{\Lambda_0} $ and $ W_{P_E^n(\theta)\Lambda_0}=W_{\Lambda_n(E)} $. This path determines a lifting path in $ \R\times\mathbf{SU}(m) $: $ (x^t_E,S^t_E) $, where $ x^0_E=0 $ and $ x^n_E=x_{A_{E,n}(\theta)} $. Moreover, \( mx_E^t \) serves as an argument function of \( \det W_{P_E^t(\theta)\Lambda_0} \), which is joint continuous in \( E \), \( t \) (and $ \theta $).
	Then \eqref{rot_f} is equivalent to 
	\begin{equation} 
		\begin{aligned}
			\op{rot}_f(T,A_E)=
		\lim_{N\to +\infty} \frac{mx_E^N- mx_E^0}{N}=
		\lim_{N\to +\infty} \frac{\arg \det W_{A_{E,N}(\theta)\Lambda}}{N}.
		\end{aligned}
	\end{equation}

	\begin{lemma}\label{lem:mono1}
		For any given Lagrangian frame $ \begin{psmallmatrix}
			X\\Y
		\end{psmallmatrix} $, let 
		\[
			\Lambda(E,n)=A_{E,n}(\cdot) \begin{pmatrix}
				X\\Y
			\end{pmatrix}=\begin{pmatrix}
				X_+(E,n)\\
				Y_+(E,n)
			\end{pmatrix},
		\]  
		the continuous argument functions of the eigenvalues of $ W_{\Lambda_n(E)} $, $ \{\varphi_j(E,n)\}_{j=1}^m $ is non-increasing as $ E $ increases. 
	\end{lemma}
	\begin{proof}
		As a starting point, we take the following lemma from \cite{HOWARD2016}:
		\begin{lemma}[\cite{HOWARD2016}, Lemma 3.11]\label{lem:monotone}
			Let $ W(\tau) $ be a smooth family of unitary $ n \times n $ matrices on some interval $ I $, and suppose $ W(\tau)  $ satisfies the differential equation $ \frac{\partial W(\tau)}{\partial \tau}=i W(\tau) \Omega(\tau) $, where $ \Omega(\tau) $ is continuous, self-adjoint and negative definite. Then the eigenvalues of $ W(\tau) $ move strictly clockwise on the unit circle as $ \tau $ increases.
		\end{lemma}
		\begin{remark}
			Note that if $ \Omega(\tau) $ is merely semi-negative definite, then $ W(\tau) $ move clockwise on the unit circle as $ \tau $ increases but not strictly. 
		\end{remark}
		A direct computation shows that, (see \cite{howard17maslov} for details)
		\begin{equation} \label{eq:W-derivative}
			\frac{\partial W_{\Lambda(E,n)}}{\partial E}=i W_{\Lambda(E,n)} \Omega(E,n),
		\end{equation} 	
		where $ \Omega(E,n)=2\left[\pa{(X_+-i Y_+ )^{-1}}^* (X_+^*\frac{\partial Y_+}{\partial E}-Y_+^*\frac{\partial X_+}{\partial E}) (X_+-i Y_+  )^{-1} \right](E,n)  $, and noting that 
	\begin{equation} 
		\begin{aligned}
			X_+^*(E,n)&\frac{\partial Y_+(E,n)}{\partial E}-Y_+^*(E,n)\frac{\partial X_+(E,n)}{\partial E}=-\begin{pmatrix}
				X^* &Y^*
			\end{pmatrix}A_{E,n}^*(\cdot) J \pa{\frac{\partial}{\partial E}A_{E,n}(\cdot)} \begin{pmatrix}
				X\\Y
			\end{pmatrix}\\
			=&-\sum_{k=1}^{n}\begin{pmatrix}
				X_{+}^*(E,k-1) &Y_{+}^*(E,k-1)
			\end{pmatrix}A_{E}^*(T^{k-1}\cdot) J \pa{\frac{\partial}{\partial E}A_{E}(T^{k-1}\cdot)} \begin{pmatrix}
				X_{+}(E,k-1)\\Y_{+}(E,k-1)
			\end{pmatrix}\\
			=&-\sum_{k=1}^n \pa{C^{-1}X_+(E,k-1)}^* C^{-1}X_+(E,k-1)\\
		\end{aligned}
	\end{equation} 
	 is continuous, self-adjoint and semi-negative definite. The result follows.
\end{proof}

By Lemma \ref{lem:mono1}, we immediately have
\begin{corollary}\label{cor:monotone}
	For any fixed $ N\in \N_+ $, $ x_E^N $ is non-increasing in $ E $.
\end{corollary}

\begin{proposition}\label{prop:eigen}
	For every finite interval $ \Lambda=[a,b]\subseteq \Z $, every eigenvalue of $ H^\Lambda $ satisfies 
	\begin{equation} 
		\text{the geometric multiplicity}=\text{the algebraic multiplicity}\leq m.
	\end{equation} 
	Moreover, we have $ \det U_z(N+1)=\det (z-H^N) $, and if $ z $ is an eigenvalue of $ H^N $  with multiplicity $ k $ if and only if 
	\begin{equation} 
		\dim\ker U_z(N+1)=k.
	\end{equation} 
\end{proposition}
\begin{proof}
	It suffices to consider the case $ H^N $, $ N\geq 1 $. Note that $ H^N $ is Hermitian, every eigenvalue of $ H^N $ satisfies 
	\begin{equation*} 
		\text{the geometric multiplicity}=\text{the algebraic multiplicity}.
	\end{equation*}  
	Let $ \tilde{U}_z=(U_z(1)\,U_z(2)\, \cdots\, U_z(N))^T $, $ \tilde{V}_z=(V_z(1)\,V_z(2)\, \cdots\, V_z(N))^T $, we obtain
	\begin{eqnarray*}
		(H^N-z)\tilde{U}_z&=&-CU_z(N+1)\delta_N.\\
		(H^N-z)\tilde{V}_z&=&-C^*\delta_1-CV_z(N+1)\delta_N.
	\end{eqnarray*}
	Thus, if $ \tau_z(N):=\dim \ker U_z(N+1)\geq 1 $,
	then $ z $ is an eigenvalue of $ H^N $, and $ \tau_z $ is the multiplicity of $ z $.
	On the other hand, suppose $ \bf{x} $ is an eigenvector of $ H^N $ with corresponding eigenvalue $ z $, i.e., $ (H^N-z)\textbf{x}=0 $. Note that $ \textbf{x}: [1,N]\to \C^{m}$ should be represented as 
	\begin{equation*} 
		\textbf{x}=\begin{pmatrix}
			U_z(1)&V_z(1)\\
			\vdots&\vdots\\
			U_z(N)&V_z(N)
		\end{pmatrix}\begin{pmatrix}
			\bf{k_1}\\
			\bf{k_2}
		\end{pmatrix},
	\end{equation*}  
	with $ {\bf{k_i}}=(k_{i,1}, k_{i,2}, \cdots,  k_{i,m})^T $, for $ i=1,2 $, and $ ({\bf{k_1}},{\bf{k_2}})^T\neq 0 $. 
	
	We claim that $ \textbf{k}_2=0 $, and thus $ \textbf{k}_1\neq 0 $. Indeed, 
	for $ N=1 $, one can easily check that 
	\begin{equation*} 
		(H^N-z)\textbf{x}=-CU_z(2)\textbf{k}_1,
	\end{equation*} 
	and for $ N\geq 2 $,
	\[
		(H^N-z)\textbf{x}=(H^N-z)\begin{pmatrix}
			\tilde{U}_z &\tilde{V}_z
		\end{pmatrix}\begin{pmatrix}
			\bf{k_1}\\
			\bf{k_2}
		\end{pmatrix}=\begin{pmatrix}
			-C^*{\bf{k}_2}\\
			0\\
			\vdots\\
			0\\
			-C(U_z(N+1){\bf{k}_1}+V_z(N+1){\bf{k}_2})
		\end{pmatrix}, 
	\] 
	so there must be $ \textbf{k}_2=0 $. Then $ (H^N-z)\bf{x}=0 $ is equivalent to 
	\begin{equation*} 
		\tau_z(N)=\dim \ker U_z(N+1)\geq 1.
	\end{equation*} 
	and the multiplicity of $ z $ is exactly $ \tau_z(N) $. Moreover, since $ \det U_z(N+1) $ and $ \det (z-H^N) $ are both monic polynomials of degree $ mN $ with the same roots, the result follows.
\end{proof}

\begin{lemma}\label{lemWlambda}
	Let $ \Lambda=\begin{bmatrix}
		X\\Y
	\end{bmatrix} $ be a Lagrangian subspace. Then we have
	\begin{enumerate}
		\item 
		$ \det W_\Lambda=\det(X+i Y)\det(X^*+i Y^*) \det (X^*X+Y^*Y)^{-1}$. 
		\item $ W_\Lambda z=z $ if and only if $ Y^*z=0 $. In particular, the eigenspace of $ W_\Lambda $ associated to $ 1 $ agrees with the kernel of $ Y $. \label{lemWlambda:2}
	\end{enumerate}
\end{lemma}
\begin{proof}
	By the symmetry $ X^*Y=Y^*X $, one can deduce 
	\[
		W_\Lambda=(X+i Y)(X^*X+Y^*Y)^{-1}(X^*+iY^*),
	\]
	Then the result follows immediately.
\end{proof}

	Let $ E_1\geq E_2\geq \cdots\geq E_{mN} $ be the eigenvalues of $ H_\theta^N $ (count with multiplicities), by Proposition \ref{prop:eigen}, we have $ E=E_j $ for some $ j=1,\dots,mN $ if and only if 
	\[
		\Lambda_{N+1}(E)=A_{E,N+1}(\theta)\begin{bmatrix}
			I_m\\0
		\end{bmatrix}=\begin{bmatrix}
			CU_E(N+2)C^{-1}\\
			U_E(N+1)C^{-1}
		\end{bmatrix}
	\] 
	satisfies $ \dim\ker U_E(N+1)=\tau_E(N)\geq 1 $, and $ E_j $ is an eigenvalue with multiplicity $ \tau_E(N) $. Then by (\ref{lemWlambda:2}) of Lemma \ref{lemWlambda}, this is also equivalent to 
	$ 1 $ is an eigenvalue of $ W_{\Lambda_{N+1}(E)} $ with multiplicity $ \tau_E(N) $, and thus if and only if the continuous argument functions of the eigenvalues of $ W_{\Lambda_{N+1}(E)} $, the unordered $ m $-tuple $ \{\varphi_1(E,N+1),\dots,\varphi_m(E,N+1)\} $ contains exactly $ \tau_E(N) $ integers.

	For any $ E>E_{1} $, (\ref{lemWlambda:2}) of Lemma \ref{lemWlambda} implies that for any $ j=1,\dots,m $, there exists $ p_j^{N+1}\in\Z $ such that
	\[
		\varphi_j(E,N+1)\in (p_j^{N+1},p_j^{N+1}+1).
	\] 
	Actually, we have 
\begin{lemma}
For every $ n\geq 1 $, every $ \theta $, every $ j=1,\dots, m $, the limit $ \lim\limits_{E\to+\infty}\varphi_j(E,n)=p_j^{n}$. Moreover, $ \lim\limits_{E\to+\infty} m x_E^n=mx_E^0=0 $, the limit is uniform in both $ n $ and $ \theta $.
\end{lemma}

\begin{proof}
Let \(U\) be the set of all Lagrangian subspaces \(\Lambda\) that can be written as $ \begin{psmallmatrix}
	I \\
	\mathcal{K}(\Lambda)
	\end{psmallmatrix} $.
Here, \(\mathcal{K}\) is a homeomorphism from \(U\) to \(H_m\), where \(H_m\) denotes the space of \(m \times m\) self-adjoint matrices.

Assume that a Hermitian symplectic matrix \(P\) maps a neighborhood of \(\Lambda_0\) into \(U\). Then it induces a map \(\hat{T}(P)\) from a neighborhood of zero in \(H_m\) into \(H_m\). If both \(\hat{T}(A)\hat{T}(B)\) and \(\hat{T}(AB)\) are well-defined, then we have $ \hat{T}(AB) = \hat{T}(A)\hat{T}(B) $
on their common domain of definition.  Actually, we only consider $V$, $\hat{T}(A)V$, $\hat{T}(B)V$, and $\hat{T}(AB)V$ for $V \in H_m$ small enough. As we will see, for sufficiently large $E$, all of these expressions are well-defined.

Recall \eqref{eq:def_Gi}, and assume that 
\[
\sup_{t \in [0,1]} \|V(t)\| < M_1, \quad \|B(\theta)\|_0 < M_2.
\]  
For each $K\in H_m$ with $\|K\|<\pa{2\|CC^*\|}^{-1} $, we have  $\hat{T}\left( G_1\right)K =K(I-CC^* K)^{-1}$, and therefore $\|\hat{T}(G_1)K\|\leq 2\|K\|$.
	Similarly, $\|\hat{T}(G_1^{-1}) K\|\leq 2\|K\|$.
By a direct calculation,
	\[
	\hat{T}(G_2)K=(t+V(t)^* K V(t))\big(I+t(E-B(\theta)-CC^*-I)(t+V(t)^* K V(t))\big)^{-1}.
	\]
Assume that $\|K\| \leq a/(M_1^2E)$, and let $D=E-B(\theta)-CC^*-I$. Then, 
 \[I+t(E-B(\theta)-CC^*-I)(t+V(t)^* KV(t))=\big(I+(tDV(t)^*KV(t))(I+t^2D)^{-1}\big)(I+t^2D).\]
Therefore, 
\[
\|\hat{T}(G_2)K\| \leq \|t+V(t)^*KV(t)\|\, \|(1+t^2D)^{-1}\|\,  \| (I+(tDV(t)^*KV(t))(I+t^2D)^{-1})^{-1}\|.
\]
For $E$ large enough (depending only on $ M_1 $, $ M_2 $), we have
$$ \|(1+t^2D)^{-1}\| \leq 1/(1+Et^2/2),$$ 
\[
	\| (I+(tDV(t)^*KV(t))(I+t^2D)^{-1})^{-1}\|\leq 1/(1-2at(1+Et^2/2 )^{-1}).
\] 
It follows that
\[
\|\hat{T}(G_2(t))K\| \leq  \frac{t+a/E}{1+Et^2/2-2at}.
\]
Note that $t/(1+Et^2/2-2at)\leq t/(\sqrt {2E}t-2at) \leq 1/\sqrt E$ and $\frac{a/E}{1+Et^2/2-2at} \leq \frac{a/E}{1-2a^2/E} \leq \frac{a}{2E}$.
So, for $E>a^2$ and large enough, $\|\hat{T}(G_2(t))K\| \leq 1/\sqrt E$ and $\|\hat{T}(G_2(1))K\| \leq \frac{1+a/E}{1+E/2-2a}\leq 8/E$.

Now choose $ a=32 M_1^2 $. Then for any $ E>a^2 $ and large enough, 
	 for all $\|K\|\le 16/E$,  we have $\| \hat T(G_1)K\| \le 32/E$, and 
	\begin{equation}
		\| \hat T(A_{E,1}(\theta)) K\|=\|\hat T(G_1^{-1}G_2(1)G_1) K\|=\|\hat T(G_1^{-1})\hat T(G_2(1))(\hat T(G_1)K)\|\le 16/E.
	\end{equation}
That is, $ A_{E,1}(\theta) $ leaves the Lagrangian subspaces satisfying $ \|K\|\leq 16/E $ invariant. Then we can conclude that for any $ \epsilon>0 $,   there exists $ E_0(\epsilon)>0 $ such that for all $ E>E_0 $, all $ t\in \R_+ $,
\[
	\|\hat{T}\pa{P_{E}^t(\theta)}(O_m)\|=\|\hat{T}\pa{P_{E}^{t-n}(\theta+n\alpha)}\hat{T}\pa{A_{E,n}(\theta)}(O_m)\| \leq \frac{4}{\sqrt{E}}<\epsilon.
\] 
Thus, for sufficiently large $E$,  $P_{E}^t(\theta) \Lambda_0$ remains in a $ O(1/\sqrt{E}) $-neighborhood of \(\Lambda_0\), then by our choice of lift,
\[
\lim_{E \to +\infty} m x_E^t = m x_E^0 = 0,
\]
with the limit uniform in \( t \) and \( \theta \). The result follows. 
\end{proof}

Since it is permitted to make a suitable choice of $ \{\varphi_j(E,n)\}_{j=1}^m $ such that $ \sum_{j=1}^m \varphi_j(E,n)=m x_E^{n} $. 
Hence, choose any $ E_0>E_1 $ and $ E_{mN+1}<E_{mN} $, for any $ E_\ell> E\geq E_{\ell+1}  $, $ \ell\in\{0,\dots,mN\} $. Corollary \ref{cor:monotone} and Proposition \ref{prop:eigen} implies that
	\begin{equation} \label{eq:k-rho}
		\ell+m\geq m x_E^{N+1}\geq \ell.
	\end{equation} 

    Note that $ \sharp\{ \text{eigenvalues of } H_{\theta}^{N} \text{ least or equal to } E  \}=mN-\ell $, \eqref{eq:k-rho} shows that
	\[
		\av{m\pa{1- \mathcal{N}_{\theta}^{N}(E)}-\frac{1}{N}m x_E^{N+1}}\leq \frac{m}{N}=O\pa{\frac{1}{N}}.
	\] 
	Take $ N\to+\infty $, we get   $ m(1-\mathcal{N}(E))=\op{rot}_f(T,A_E) $. Since any other choice of lift will result in $ \op{rot}_f(T,A_E) $ differing by a possible value of \( \int_{\Theta} p(\theta) \, d\mu \) with continuous $ p(\theta):\Theta\to\Z $, we have
	\[
		m\pa{1-\mathcal{N}(E)}=\op{rot}_f(T,A_E)\mod \mu(C(\Omega,\Z)).
	\]  

	\section{Proof of Theorem \ref{t:gablabel}}\label{sec:GLT}
	Given $ (\Theta,T,\mu) $ with $ \supp \mu=\Theta $. Combines Theorem \ref{Haropuig} with the assumption $ \supp \mu=\Theta $, we immediately have that the almost sure spectrum $ \Sigma_{\mu}=\{E\in\R: (T,A_E) \text{ is uniformly hyperbolic}\} $. 
	Let $ E\in \R\setminus\Sigma_{\mu} $, then $ (T,A_E) $ is uniformly hyperbolic, and therefore there exists an invariant Whitney splitting $ \C^m\times \Theta=\Lambda^s\oplus\Lambda^u $ such that  for all $\theta \in \Theta $, one has $ A_E(\theta) \Lambda^s (\theta) = \Lambda^s(T\theta)
	$,  $ A_E(\theta)\Lambda^u (\theta) = \Lambda^u(T \theta) $. It is well known that each fiber of continuous bundles $ \Lambda^s $ and $ \Lambda^{u} $ are Lagrangian planes; see \cite{JON+2016Nonautonomous}. 
	
	Let $(X,\overline{T},\overline{\mu})$ be the suspension of $ (\Theta,T,\mu) $. In the previous section, we have already constructed a symplectic path $ P_E^t(\theta) $, which satisfies $ P_{E}^n(\theta)=A_{E,n}(\theta) $, $ P_{E}^0(\theta)=I_m $. 
	Define 
	\[
		\Lambda^{*}(T^t\theta)=P_E^t(\theta) \Lambda^{*}(\theta),\  *=s,u.
	\] 
	It descends to give continuous bundles $ \overline{\Lambda^u} $, $ \overline{\Lambda^s} $ on $ X\to \op{Lag}( \mathbb{C}^{2m},\omega_m )  $.  
	It is easy to verify that the map
	\[
		\phi: X \to \T, \ x=[\theta,s]\mapsto \det{W_{\Lambda^u(T^s\theta)}}
	\] 
	is well-defined and continuous. Reviewing the proof in the previous sections, the fibered rotation number does not depend on the choice of initial Lagrangian plane, so we can substitute $ \Lambda_0 $ by $ \Lambda^u(\theta) $, it then follows
	\[
		m(1-\mathcal{N}(E))=\op{rot}_f(T,A_E)=\lim_{N\to\infty} \frac{\tilde{\phi}(\overline{T}^N x)}{N}, \ \overline{\mu}\text{-a.e. } x\in X  
	\] 
	where $ \tilde{\phi} $ is any lift of $ \phi $ to a map $ X\to\R $. By \cite[Theorem 3.9.13]{DF}, this gives $ m(1-\mathcal{N}(E)) $ must belongs to the Schwartzman group $  \schwartzmanGroup(\Theta,T,\mu):=\mathfrak{A}_{\overline{\mu}}(C^\sharp(X,\T)) $. Since  the image of the Schwartzman homomorphism always contains $ \Z $, see \cite[Exercise 4.10.2]{DF}, the result follows.

\section{Interval spectrum}\label{sec:interval}

\begin{proposition}\label{p.keyprop}
	Assume that $X$ is connected. Suppose $E_1, E_2 \in \R$ are such that $E_1 < E_2$ and the cocycles $(T, A_{E_1})$ and $(T, A_{E_2})$ are uniformly hyperbolic. Then either $[E_1, E_2]\cap \Sigma_0=\emptyset$, or $\Sigma_0\subset (E_1, E_2)$, or the unstable sections $\Lambda^{u}_{E_1}$ and $\Lambda^{u}_{E_2}$ are not homotopic.
	\end{proposition}

	\begin{proposition}\label{p.homotopic}
		Suppose $\Lambda^1 : \Theta \to \op{Lag}( \mathbb{F}^{2m},\omega_m )$ and $\Lambda^2 : \Theta\to \op{Lag}( \mathbb{F}^{2m},\omega_m )$ are continuous. If we have $\Lambda^1(\theta) \cap \Lambda^2(\theta)=\{0\}$ for all $\theta\in \Theta$, then $\Lambda^1$ and $\Lambda^2$ are homotopic.
		\end{proposition}
		\begin{proof}
			Note that $
			\dim \left(\Lambda_{1}(\theta)\cap \Lambda_{2}(\theta)\right)=\dim \ker(W_{\Lambda_{2}}^{-1}W_{\Lambda_{1}}^{}(\theta)-I_{m})=0 $ for all $ \theta\in\Theta $. Then by (2) of Lemma \ref{lemWlambda}, $W_{\Lambda_{2}}^{-1}W_{\Lambda_{1}}^{}(\theta)$ is actually a map from $ \Theta $ to a contractible submanifold of $ \op{Lag}( \mathbb{F}^{2m},\omega_m ) $. This means $\Lambda^1$ and $\Lambda^2$ are homotopic.
		\end{proof}

		As a corollary (of Proposition \ref{p.homotopic} and Theorem \ref{Haropuig}) we get the following simple fact that we state explicitly:
\begin{proposition}\label{p.stableunstablehomotopic}
If $E \not \in \Sigma_0$, then the unstable section $\Lambda^{u}_{E}$ and the stable section $\Lambda^{s}_{E}$ of $(T,A_E)$ are homotopic.
\end{proposition}
\begin{proposition}\label{p.3}
	Suppose $\Theta $ is connected and the cocycle $(T, A)$ has an invariant section $\Lambda : \Theta \to \op{Lag}( \mathbb{F}^{2m},\omega_m )$. Suppose that another cocycle $(T, A')$ is such that it has the same invariant section. Then
	$$
	\op{rot}_f(T, A)=\op{rot}_f(T, A') \!\!\! \mod \Z.
	$$
	\end{proposition}
	
	\begin{proof}
	Notice that different lifts differ a continuous integer-valued function. Since $\Theta $ is connected, the result follows.
	\end{proof}
	
	\begin{proposition}\label{p.2}
	Suppose that $\Theta $ is connected, $(T, A)$ is an $\rmm{HSp}(2m, \C)$ cocycle that is homotopic to the identity,  $\alpha : \Theta \to \mathbb{R}$ is continuous, and the cocycle $(T, C)$ is given by $C(\theta) = R_{-\alpha(T\theta)}A(\theta)R_{\alpha(\theta)}$,
	where $R_* \in \rmm{HSp}(2, \mathbb{C})\cap \rmm{SU}(2m,\C)$.
	Then $\op{rot}_f(T, C)=\op{rot}_f(T, A) \!\!\! \mod \Z$.
	\end{proposition}
	\begin{proof}
		The proof follows a similar line as the proofs of Lemma \ref{independent}. Define $ T^t\theta:=(1-t)\theta+t T\theta $ for any $ t\in(0,1) $, consider any symplectic path $ A^t $ induced by homotopy, and set $ C^t $ be any symplectic path satisfying $ C^t(\theta)=R_{-\alpha(T^t\theta)}A^t(\theta)R_{\alpha(\theta)} $.
For any $ R_{\alpha(\theta)}\in \rmm{HSp}(2, \mathbb{C})\cap \rmm{SU}(2m,\C) $, there exist $ R_1(\theta), R_2(\theta)\in \mathbf{U}(m) $ such that
\[
	{\overset{\circ}{R_{\alpha(\theta)}}}=\begin{pmatrix}
		R_1(\theta)&\\ & R_2(\theta)
	\end{pmatrix}.
\] 
Let $ \Lambda $ be any Lagrangian subspace, and $ \Lambda_1=R_{\alpha(\theta)}\Lambda $. By \eqref{eq:W1234} and performing a direct calculation, we obtain
\[
	\begin{aligned}
		W_{\Lambda}^{-1}W_{C^t\Lambda}^{}&=R_1(\theta)^{-1}W_{\Lambda_1}^{-1}R_2(\theta)R_2(T^t\theta)W_{A^t\Lambda_1}R_1(T^t \theta)^{-1}.
	\end{aligned}
\] 
It is clear that the argument differences between $ \det W_{C^t\Lambda}^{}W_{\Lambda}^{-1} $ and $ \det W_{A^t\Lambda_1}^{}W_{\Lambda_1}^{-1} $ is bounded by a constant, since $ \Theta $ is compact. Then the result follows from Lemma \ref{independent}.
	\end{proof}
	
	Together, Propositions \ref{p.2} and \ref{p.3} imply the following:
	
	\begin{proposition}\label{p.homotopicsamegap}
	Assume that $\Theta$ is connected. If the unstable sections $\Lambda^{u}_{E_1}$ and $\Lambda^{u}_{E_2}$ are homotopic, then $\op{rot}_f(E_1)=\op{rot}_f(E_2)\!\! \mod \Z$.
	\end{proposition}
	
	\begin{proof}
	If $\Lambda^{u}_{E_1}$ and $\Lambda^{u}_{E_2}$ are homotopic, the cocycle $(T, A_{E_1})$ is conjugate to a cocycle for which the section $\Lambda^{u}_{E_2}$ is invariant (and which, due to Proposition~\ref{p.3}, has the same rotation number). On the other hand, note that any two Lagrangian subspaces can be transferred by an action in $ \rmm{HSp}(2, \mathbb{C})\cap \rmm{SU}(2m,\C) $, due to Proposition~\ref{p.2} it also must have the same rotation number as $(T, A_{E_2})$.
	\end{proof}
	
	\begin{proof}[Proof of Proposition~\ref{p.keyprop}]
	Given $E_1, E_2 \in \R$ with $E_1 < E_2$ so that the cocycles $(T, A_{E_1})$ and $(T, A_{E_2})$ are uniformly hyperbolic, it follows from Theorem \ref{Haropuig} that $E_1$ and $E_2$ belong to the complement of $\Sigma_0$. If they belong to the same connected component of $\Sigma_0^c \cap \R$ (i.e., the same gap), then we have $[E_1, E_2] \cap \Sigma_0=\emptyset$. If $E_1 < \min \Sigma_0$ and $E_2 > \max \Sigma_0$, then $\Sigma_0\subset (E_1, E_2)$. Otherwise, we must have that $E_1$ and $E_2$ belong to different gaps of $\Sigma_0$, one of which is bounded. In this case, it follows from Theorem~\ref{Mainthm} and Proposition~\ref{p.homotopicsamegap} that the unstable sections $\Lambda^{u}_{E_1}$ and $\Lambda^{u}_{E_2}$ are not homotopic.
	\end{proof}
	We present the formal argument in the proof of the following statement:

\begin{lemma}\label{l.main1}
If $E \in \Sigma_1^c$, then $E - S \subseteq \Sigma_0^c$ and all $E' \in E - S$ have homotopic unstable sections with respect to the unperturbed cocycle at energy $E'$.
\end{lemma}

\begin{proof}
We show the contrapositive. That is, if $E - S \not\subseteq \Sigma_0^c$ or if $E - S \subseteq \Sigma_0^c$ and there are $v, v' \in S$ that have non-homotopic unstable sections at energies $E - v$ and $E - v'$ with respect to the respective unperturbed cocycles, then $E \in \Sigma_1$.

Consider first the case $E - S \not\subseteq \Sigma_0^c$. Then, there is $v \in S$ such that $E - v \in \Sigma_0$. This shows that the constant realization $W_\omega \equiv v\cdot I_m $ is such that $E \in \sigma(H_{\theta} + W_\omega)\subseteq \Sigma_1 $ for almost every $\theta \in \Theta $ with $\sigma(H_\theta) = \Sigma_0$. 

In the other case, $E - S \subseteq \Sigma_0^c$ and there are $v, v' \in S$ such that $E- v$ and $E - v'$ have non-homotopic unstable sections with respect to the respective unperturbed cocycles. Consider the random realization
$$
W_\omega(n) = \begin{cases} v\cdot I_m & n \in \Z_-, \\ v'\cdot I_m & n \in \Z_+. \end{cases}
$$
Since the stable and unstable sections of the unperturbed cocycle for fixed energy are homotopic by Proposition~\ref{p.stableunstablehomotopic}, by assumption we have that the unstable section for energy $E - v$ and the stable section for energy $E - v'$ are non-homotopic (they exist due to $E - S \subseteq \Sigma_0^c$). By Proposition~\ref{p.homotopic} there exists $\theta \in \Theta $ such that $\Lambda^u_{E-v}(\theta) = \Lambda^s_{E-v'}(\theta)$. This shows that the generalized Schr\"odinger operator with diagonal block $B_\theta + W_\omega$ for these particular choices for $\theta $ and $\omega$ possesses an exponentially localized eigenvector at energy $E$. Thus, we have $E \in \sigma(H_\theta + W_\omega)$, and hence by Proposition \ref{prop:smallerspec}, we have $E \in \Sigma_1$.
\end{proof}

\begin{lemma}\label{l.main}
Assume that $\Theta$ is connected. Then, $E \in \Sigma_1^c$ if and only if either $E - S$ is contained in a gap of $\Sigma_0$ or $\Sigma_0$ is contained in a gap of $E - S$.
\end{lemma}

\begin{proof}
For the first direction we suppose that $E \in \R$ is such that neither $E - S$ is contained in a gap of $\Sigma_0$, nor $\Sigma_0$ is contained in a gap of $E - S$. We need to show that $E \in \Sigma_1$.

One possibility is that $E - S$ intersects $\Sigma_0$. By the argument given above in the proof of the first case of Lemma~\ref{l.main1}, it follows that $E \in \Sigma_1$, as desired.

The other possibility is that neither set is contained in a gap of the other, but they still have empty intersection. In this case one can find $v,v' \in S$ such that $E - v$ and $E - v'$ belong to different gaps of $\Sigma_0$, one of which must be an interior (i.e., bounded) gap. Proposition~\ref{p.keyprop} now shows that the unstable sections at these two energies are non-homotopic. Thus, by Lemma~\ref{l.main1} we find $E \in \Sigma_1$, again as desired.

\medskip

For the reverse direction we suppose that $E \in \R$ is such that either $E - S$ is contained in a gap of $\Sigma_0$, or $\Sigma_0$ is contained in a gap of $E - S$. We need to show that $E \in \Sigma_1^c$. 
It is clear that $\Sigma_1 \subseteq \mathrm{ch}(S)+\Sigma_0$; see \cite[Lemma 3.1]{takase2021spectra}. Thus, if $E - S$ is contained in a gap of $\Sigma_0$, then $E - \mathrm{ch}(S)$ is also contained in a gap of $\Sigma_0$, and hence $E$ cannot be in $\mathrm{ch}(S)+\Sigma_0$, so $E\not\in \Sigma_1$.

To see the other implication, we can assume without loss of generality that $r := \max \Sigma_0 = - \min \Sigma_0$ (otherwise shift appropriately and subsume the necessary translate in $S$). By self-adjointness, we therefore must have
\begin{equation}\label{e.Hx}
\|H_\theta\| = r \ \text{ for}\ \ \mu\text{-almost every}\ \ \theta \in \Theta.
\end{equation}
Arguing in a similar way as before, the addition of $H_\theta $ to $W_\omega$ can shift the edge of a spectral gap by no more than $r$ (for $\mu$-almost every $\theta \in \Theta$). Since $S$ is also the spectrum of the operator $W_\omega$ for $\tilde \mu$-almost every $\omega \in \Omega$, it follows that $E \in \Sigma_1^c$ if $\Sigma_0$ is contained in a gap of $E - S$. Indeed, if $\Sigma_0$ is contained in a gap of $E-S$, then $\mathrm{ch}(\Sigma_0)$ is contained in the same gap of $E - S$ as well, which implies that
$$
E \not \in S + \mathrm{ch}(\Sigma_0) = S + [-r, r].
$$
Due to (\ref{e.Hx}), this implies that $E \in \Sigma_1^c$.
\end{proof}

\begin{proof}[Proof of Theorem~\ref{Thm:bigstar}]
Consider first the case where $\op{diam}(S) \le \op{diam}(\Sigma_0)$. Then, by Lemma~\ref{l.main}, $E \in \Sigma_1^c$ if and only if $E - S$ is contained in a gap of $\Sigma_0$ (as the other case is impossible). But this in turn is equivalent to the statement that $E - \mathrm{ch}(S)$ is contained in a gap of $\Sigma_0$. It follows that an $E \in \R$ obeys $E \notin \Sigma_1$ if and only if $E \notin \Sigma_0 + \mathrm{ch}(S)$, whence $\Sigma_1 = \Sigma_0 \bigstar S$ in this case.

In the case where $\op{diam}(S) > \op{diam}(\Sigma_0)$, we argue similarly. By Lemma~\ref{l.main}, $E \in \Sigma_1^c$ if and only if $\Sigma_0$ is contained in a gap of $E - S$. This in turn is equivalent to the statement that $\mathrm{ch}(\Sigma_0)$ is contained in a gap of $E - S$. It follows that an $E \in \R$ obeys $E \notin \Sigma_1$ if and only if $E \notin \mathrm{ch}(\Sigma_0) + S$, whence $\Sigma_1 = \Sigma_0 \bigstar S$ in this case as well.
\end{proof}

\noindent {\bf Acknowledgements.} The authors sincerely thank Professor Jiangong You and Professor Qi Zhou for their interest and productive discussions. The first author also thanks Siyuan Chen for bringing \cite{universalcover} to his mind and Jake Fillman for useful suggestions. X. Li was partially supported by NSFC grant (123B2005), Nankai Zhide Foundation, and an AMS-Simons Travel Grant.
L. Wu was supported
by NSFC grant (12171281). 

	\appendix
	\renewcommand{\appendixname}{Appendix~\Alph{section}}

\section{Almost invariance under conjugations}
Let \( (\Theta, T, \mu) = (\mathbb{T}^d, \alpha, \operatorname{Leb}) \). For later use in the study of quasi-periodic generalized Schrödinger operators acting on \( \ell^2(\mathbb{Z}, \mathbb{C}^m) \), we consider the variation of the fibered rotation number under conjugations.

Again, recall the singular value decomposition for any \( \Aa \in \mathbf{U}(m,m) \). There exist unitary matrices \( U_1, U_2, V_1, V_2 \in \mathbf{U}(m) \) and a positive diagonal matrix \( \Gamma \) such that
\[
\Aa = \begin{pmatrix}
U_1 & 0 \\ 
0 & U_2
\end{pmatrix}
\begin{pmatrix}
\cosh(\Gamma) & \sinh(\Gamma) \\
\sinh(\Gamma) & \cosh(\Gamma)
\end{pmatrix}
\begin{pmatrix}
V_1 & 0 \\ 
0 & V_2
\end{pmatrix}.
\]
Define
\[
\begin{aligned}
\overset{\circ}{P} &= \begin{pmatrix}
U_1 & 0 \\ 
0 & U_2
\end{pmatrix}
\begin{pmatrix}
\cosh(\Gamma) & \sinh(\Gamma) \\
\sinh(\Gamma) & \cosh(\Gamma)
\end{pmatrix}
\begin{pmatrix}
U_1^* & 0 \\ 
0 & U_2^*
\end{pmatrix}, \\
\overset{\circ}{U} &= \begin{pmatrix}
U_1 & 0 \\ 
0 & U_2
\end{pmatrix}
\begin{pmatrix}
V_1 & 0 \\ 
0 & V_2
\end{pmatrix}.
\end{aligned}
\]
Then, \( \overset{\circ}{P} \in \mathbf{U}(m,m) \) is symmetric and positive definite, and \( \overset{\circ}{U} \in \mathbf{U}(m,m) \) is unitary. This is called the polar decomposition, and it is unique.

Given 
\[
B: 2\T^d \to \rmm{HSp}(2m,\C),
\]
let 
\[
\overset{\circ}{B}(\theta)=\overset{\circ}{P}(\theta)\overset{\circ}{U}(\theta)
\]
be its polar decomposition. We can immediately see that \(\overset{\circ}{B}(\theta)\) is homotopic to \(\overset{\circ}{U}(\theta)\). Then, we can define the degree of \(B\) as below.

\begin{definition}
	If \(\overset{\circ}{B}(\theta)\) is homotopic to 
	\[
	\overset{\circ}{U}(\theta): \T^d \to \mathbf{U}(m,m) \cap \mathbf{U}(2m,\C),
	\]
	where
	\[
	\overset{\circ}{U}(\theta) = \begin{pmatrix}
		U_1(\theta) & 0 \\ 
		0 & U_2(\theta)
	\end{pmatrix},
	\]
	then the degree of \(B\), denoted by \(\deg B\), is defined as \(r\in\Z^d\) such that \(\det\bigl(U_1(\theta)^{-1}U_2(\theta)\bigr)\) is homotopic to \(e^{2\pi i \langle r, \theta\rangle}\).
\end{definition}

\begin{proposition}
	Let \((\Theta,T,\mu)=(\T^d,\alpha,\op{Leb})\). If \(A:\T^d\to \rmm{HSp}(2m,\C)\) is continuous and homotopic to the identity and if \(B:\T^d\to \rmm{HSp}(2m,\C)\) is continuous and of degree \(r\in \Z^{d}\), then
	\[
	\op{rot}_f\Bigl((0,B)^{-1}\circ (\alpha,A)\circ (0,B)\Bigr)=\op{rot}_f\bigl((\alpha,A)\bigr)-\langle r,\alpha\rangle \mod \Z.
	\]
\end{proposition}

\begin{remark}
	If \(B(\cdot)\) is only defined on \(2\T^d\) (and not on \(\T^d\)), the map \(\theta\mapsto B(2\theta)\) is of degree \(r\in \Z^{d}\), and \(B(\cdot+\alpha)^{-1}A(\cdot)B(\cdot)\) is defined on \(\T^d\). Then,
	\[
	\op{rot}_f\Bigl((0,B)^{-1}\circ (\alpha,A)\circ (0,B)\Bigr)=\op{rot}_f\bigl((\alpha,A)\bigr)-\frac{\langle r,\alpha\rangle}{2} \mod \Z.
	\]
\end{remark}

\begin{proof}
	We first consider the case when \( B \) is homotopic to the identity. By the previous discussion, there exist maps \( \tilde{F}_{0,B} \) and \( \tilde{F}_{0,B^{-1}} \) which are lifts of \( (0,B) \) and \( (0,B)^{-1} \), respectively. It is clear that 
	\[
		\op{rot}_f\bigl((0,B)^{-1}\circ (\alpha,A)\circ (0,B)\bigr) = \op{rot}_f\bigl((\alpha,A)\bigr) \mod \Z.
	\]
	
	Now, if \( B \) is of degree \( r\in \Z^{d} \), then \( \overset{\circ}{B} \) can be written as \( \overset{\circ}{B}(\cdot)=\overset{\circ}{P}(\cdot)\overset{\circ}{U}(\cdot) \) with \( \overset{\circ}{P}(\cdot) \) homotopic to the identity. It is then enough to check the proposition for \( U(\cdot) \) 
	in place of \( B(\cdot) \), up to homotopy, we may further assume that \(\det\bigl(U_1(\theta)^{-1}U_2(\theta)\bigr)=e^{2\pi i\langle r,\theta\rangle}\).   It is not difficult to check that a lift for 
	\[
		\bigl(\alpha,\, U(\cdot+\alpha)^{-1}A(\cdot)U(\cdot)\bigr)
	\]
	is given by 
	\[
		p_2\circ\tilde{G}(\theta,x,S)= x + f\Bigl(\tilde{F}_{\alpha,A}\Bigl(\theta, x + \frac{\langle r,\theta\rangle}{m}, \bar{S}\Bigr)\Bigr) - \frac{\langle r,\alpha\rangle}{m}.
	\]
	By the uniformity of the corresponding rotation numbers, one obtains 
	\[
		\op{rot}_f\Bigl(\alpha,\, U(\cdot+\alpha)^{-1}A(\cdot)U(\cdot)\Bigr) = \op{rot}_f\bigl((\alpha,A)\bigr) - \langle r,\alpha\rangle \mod \Z,
	\]
	which completes the proof.
\end{proof}

\section{Aubry duality and gap labels}\label{apx.aubryD}
Here, we provide an explanation for Remark \ref{fractional-label}. As a typical example, 
let \( v = \sum_{k=-d}^{d} \hat{v}_k e^{2\pi i k\theta} \) be a trigonometric polynomial of degree \( d \). The dual operator \( L_{v,\alpha,\theta} \) of the quasi-periodic Schrödinger operator \( H_{v,\alpha,\theta} \), acting on \( \ell^2(\mathbb{Z}) \), is given by  
\[
(H_{v,\alpha,\theta} u)_n = u_{n-1} + u_{n+1} + v(\theta + n\alpha) u_n.
\]  
The corresponding dual operator takes the form  
\[
(L_{v,\alpha,\theta} u)_n = \sum_{k=-d}^{d} \hat{v}_k u_{n+k} + 2\cos 2\pi (\theta + n\alpha) u_n.
\]
If we define  
\[
C = \begin{pmatrix}  
\hat{v}_d & \cdots & \hat{v}_1 \\  
0 & \ddots & \vdots \\  
0 & 0 & \hat{v}_d  
\end{pmatrix},  
\quad B_{\theta}(0) = f(\theta) = \begin{pmatrix}  
2\cos 2\pi(\theta_{d-1}) & \hat{v}_{-1} & \cdots & \hat{v}_{-d+1} \\  
\hat{v}_1 & \ddots & \ddots & \vdots \\  
\vdots & \ddots & 2\cos 2\pi(\theta_1) & \hat{v}_{-1} \\  
\hat{v}_{d-1} & \cdots & \hat{v}_1 & 2\cos 2\pi(\theta_{0})  
\end{pmatrix},
\]  
where \( \theta_j = \theta + j\alpha \),  
then we find that \( L_{v,\alpha,\theta} \) is exactly the operator \( H_\theta \) acting on \( \ell^2(\mathbb{Z}, \mathbb{C}^d) \) of the form \eqref{eq:H2}.  

Furthermore, the cocycle induced by the eigenvalue equation  
\[
L_{v,\alpha,\theta} u = E u
\]  
can be denoted by \( (\alpha, L_{E,v}) \), where
\[
L_{E,v}(\theta) = \frac{1}{\hat{v}_d}  
\begin{pmatrix}  
-\hat{v}_{d-1} & \cdots & -\hat{v}_1 & E - 2\cos 2\pi \theta - \hat{v}_0 & -\hat{v}_{-1} & \cdots & -\hat{v}_{-d+1} & -\hat{v}_{-d} \\  
\hat{v}_d &  &  &  &  &  &  &  \\  
& & & & & & &  \\  
& & & & & & &  \\  
& & & \ddots & & & &  \\  
& & & & & & &  \\  
& & & & & & &  \\  
& & & & & & \hat{v}_d &  
\end{pmatrix}.
\]  

Then, one can verify that $ P L_{E,v}P^{-1}\in \op{HSp}(2m,\C) $ and 
\begin{equation} \label{strip2}  
L_{E,v}(\theta + (d-1)\alpha) \cdots L_{E,v}(\theta) =: L_{d,E,v}(\theta) =  
\begin{pmatrix}  
C^{-1} (E - f(\theta)) & -C^{-1} C^* \\  
I_d & O_d  
\end{pmatrix} = \widehat{A}_E(\theta).  
\end{equation}  

Notice that \eqref{strip2} implies  
\[
(\alpha, L_{E,v})^d = (d\alpha, L_{d,E,v}) = (d\alpha, \widehat{A}_E),
\]  
and that \( (\alpha, L_{E,v}) \) is uniformly hyperbolic if and only if \( (d\alpha, \widehat{A}_E) \) is uniformly hyperbolic.  

Now, if we instead consider the cocycle \( (\alpha, L_{E,v}) \) as in section \ref{sec:GLT}, it follows that for any \( E \in \mathbb{R} \setminus \Sigma \),  
\[
d(1 - \mathcal{N}(E)) =\op{rot}_f (d\alpha, A_E) = d \operatorname{rot}_f(\alpha, P^{-1} L_{E,v}P) \in d \cdot \schwartzmanGroup(\mathbb{T}, \alpha, \operatorname{Leb}) = d(\mathbb{Z} \alpha + \mathbb{Z}).
\]  
That is,  $ \mathcal{N}(E) \in \mathbb{Z} \alpha \cap [0,1] $.

\begin{remark}
	Let \( \mathcal{N}^H(E) \) and \( \mathcal{N}^L(E) \) be the integrated density of states (IDS) for the operators \( H_{v,\alpha,\theta} \) and \( L_{v,\alpha,\theta} \), respectively. We can easily check that \( \mathcal{N}^L(E) = \mathcal{N}(E) \), and it is known that \( \mathcal{N}^L(E) = \mathcal{N}^H(E) \). Therefore, we recover the one-dimensional case.
\end{remark}

\section{An alternative proof of Example \ref{Gap-Label}}\label{apx.Ex1}
Let \( (\Theta,T,\mu)=(\T^d,\alpha,\op{Leb}) \). Fix any \( E\in \R\setminus \Sigma \); it follows that \((\alpha,A_E)\) is uniformly hyperbolic. Consider the symplectic path \( P_E^t(\theta) \) constructed in the previous section. Let 
\[
\Lambda: \T^d \to \op{Lag}( \mathbb{F}^{2m},\omega_m )
\]
be the corresponding stable Lagrangian subspace at \(\theta\); then we have
\begin{equation} \label{stable-section}
	\Lambda(\theta+\alpha)=P^{1}_E(\theta)\Lambda(\theta)=A_E(\theta)\Lambda(\theta).
\end{equation} 
Let \(\pi\) be the canonical covering map \(\R^d\to \T^d\), \(\pi(x)=\theta\). There exists a continuous map 
\[
f: x\in \R^d \mapsto  \det W_{\Lambda(\pi (x))}\in \T.
\]
By the lifting criterion, it can be lifted to \(\tilde f :\R^d \to \R\). For any \( y\in\R^d \), let 
\[
h_y(x)=\tilde f(x+y)-\tilde f(x),
\]
and it is easy to verify that \( h_y(\cdot) \) is \(\Z^d\)-periodic. Then by the unique ergodicity of the irrational rotation, the limit
\[
	\mu=\lim_{n\to +\infty}\frac{\tilde f(\theta+n\alpha)-\tilde f(\theta)}{n} 
	=\lim_{n\to +\infty} \frac{1}{n} \sum_{k=0}^{n-1} h_{\alpha}(\theta+k\alpha) 
	=\int_{\T^d} h_{\alpha}(\theta) \, d\theta
\]
exists for any \(\theta\in\T^d\).

Let \(\alpha=(\alpha_1,\alpha_2,\cdots,\alpha_d)\), and let \(\mathbf{e}_j\in \R^d\) be the canonical basis. Then we have
\[
\mu= \sum_{1\le j\le d} \int_{\T^d} h_{\alpha_j \mathbf{e}_j}(\theta) \, d\theta.
\]
For any \( s\in\R \), denote 
\[
h_j(s)=\int_{\T^d} h_{s \mathbf{e}_j}(\theta)\,d\theta.
\]
Note that \( h_j(s) \) is a continuous Cauchy function and \( h_j(1)\in \Z \). So 
\[
h_j(\alpha_j)=h_j(1)\alpha_j=k_j\alpha_j
\]
for some \( k_j\in \Z \). We conclude that  
\[
\mu= \langle k, \alpha\rangle 
\]
for some \( k\in\Z^d \).

Let  
\[
g(t,x)=\det W_{P_E^t(\pi(x))\Lambda(\pi(x))}.
\]
It can be lifted to \(\tilde g(\cdot, x):[0,1]\to \R\). Combining \eqref{stable-section} and Theorem \ref{Mainthm}, the integrated density of states can be calculated as
\[
m(1-\mathcal{N}(E))=\lim_{n\to +\infty}\frac{1}{n}\sum_{k=0}^{n-1} \Bigl(\tilde g(1,x+k\alpha)-\tilde g(0,x+k\alpha)\Bigr).
\]
Notice that \(\tilde g(1,x)-\tilde g(0,x)\) and  \(h_{\alpha}(x)\) differ at most by an integer. This gives
\[
	m(1-\mathcal{N}(E))=\mu \mod \Z \in \Z^d \alpha+\Z.
\]

\end{document}